\documentclass[12pt]{amsart}
\usepackage[final]{graphicx}
\usepackage[frame]{xypic}
\usepackage{amsfonts, epsfig, latexsym, amsmath, amsthm, eucal, 
amssymb,amsfonts,amstext,booktabs}
\usepackage{verbatim} 
\usepackage{setspace}
\usepackage{fullpage}
\usepackage{enumerate}
\newtheorem{lemma}{Lemma}

\newtheorem{theorem}{Theorem}
\newtheorem*{corollary}{Corollary}
\newtheorem{defin}{Definition}

\newtheorem*{note}{Note}

\long\def\symbolfootnotetext[#1]#2{\begingroup%
\def\thefootnote{\fnsymbol{footnote}}\footnotetext[#1]{#2}\endgroup}

\begin{document}
\doublespacing

\title{Generalized Prime Ideals of the Rings $\mathbf{C(X,Y)}$ and the Quasi-Components of $\mathbf{X}$} 
\author{H. J. Charlton}
\address{Department of Mathematics, North Carolina State University, Raleigh, North Carolina}
\dedicatory{Dedicated in the memory of Dr. Kwangil Koh}

\begin{abstract}
In the set of continuous functions C(X,Y) where Y has a topology close to being discrete, there is an equivalence relation on X which characterizes the quasi-components of X.  If Y satisfies weak algebraic conditions with a single binary operation then a stable set of functions forms an object generalizing an ideal of a ring.  Calling such sets ideals there is a concept of a  prime ideal.  The ideal of functions vanishing on a quasi-component are prime ideals of C(X,Y).  If Y has a zero set that is open then these prime ideals are min-max implying that when Y is a ring all of the prime ideals of C(X,Y) are of this form and min-max.  However this is a study of C(X,Y) and its ideals beginning with a few algebraic hypothesis on Y and adding to them as needed.   So there are conditions when a prime ideal is minimal, when the set of quasi-components is in bijective correspondence with the set of prime ideals of functions which vanish on them,  when C(X,Y) can not be like a local ring, when the prime (nill) radical is trivial, and when the corresponding Spect(C(X,Y)) is a disconnected topological space.  Example of results are; if a quasi-component of x is open then the prime ideal of functions vanishing on it is minimal independent of the zero set of Y being open;  if the zero set of Y is open then the set of all ideals of functions vanishing on quasi-components is the set of minimal prime ideals irrespective of any quasi-component being open in X.   These results imply the same for the ideals of the ring C(X,Y) when Y is a ring.  The techniques developed give a method of generating unlimited number of rings with prescribed sets of prime ideals and minimal prime ideals.
\end{abstract}

\maketitle

Let X and Y be a topological spaces and C(X,Y) the set of all continuous functions from X to Y.  Assume that X and Y both are of cardinality two or greater, $\Vert X \Vert >1$ and $\Vert Y\Vert >1$.  Let $\mathbb{Z}$, $\mathbb{Z}_{n}$, $\mathbb{Q}$, and $\mathbb{R}$ be respectively the integers, the integers mod n, the rational numbers, and the reals.  The question whether there is a relationship between the connectivity of X and the ring C(X,Y), and the observation that the rings $C(X,\mathbb{Z})$ and $C(X,\mathbb{Q})$ need not be isomorphic, motivated this study.

\begin{theorem}
If Y has a continuous binary operation denoted by $\cdot$  which may satisfy additional algebraic axioms, then the point-wise composition of $f$ and $g$, $f\cdot g$, defined by $(f \cdot g)(x)=f(x)\cdot g(x)$,  is a continuous binary operation on C(X,Y) satisfying those axioms.  When appropriate we will have additional unitary and binary operations which must be continuous.
\end{theorem}

\begin{proof}
For example the continuity of the defined binary operation in C(X,Y) would follow because the product map $<f,g>(x)=(f(x),g(x))$ from $X$ to $Y\mathsf xY$ is continuous and the binary operation in $Y$ which is a map from $Y\mathsf xY$ to Y is continuous.
\end{proof}

\begin{note}
As continuity of any operation in Y is not always needed and is evident when used, reference to continuity will usually be omitted.  Further assume the least number of algebraic operations and properties that are needed.  As a consequence most of the results do not use associativity, commutativity, or distributivity.
\end{note}

For want of a peer the presentation here is idiosyncratic.  It follows the conversation that the author was having with himself as it was being written and revised.  Consequently employing elementary ideas from several areas, the development is often tedious and elementary.  Principally these ares are algebra, algebraic geometry, analysis, and point set topology but alas not algebraic topology.

If Y is an algebraic structure then C(X,Y) has a similar algebraic structure with the usual operations on functions induced from Y.  We will be interested in a particular equivalence relation  
on X relative to its quasi-connectivity as a topological space and how this relates to the prime ideals of C(X,Y) when Y has special topological and algebraic properties. 
The equivalence classes relate to the ideals of another ring of continuous functions isomorphic to C(X,Y) which is easier to study.  These ideals have a geometric origin which can be manipulated.  The prime ideals of these C(X,Y) are of fundamental importance.  However to develop a clear insight into the origin of these ideals it was necessary to assume a strong topology on Y ("almost discrete") and start with weak algebraic properties such as a magna that need not be associative or commutative.  

First, for any equivalence relation on X, let $\Pi$ denote the set of equivalence classes, and $X/\Pi$ the topological quotient space.  Use [x] for a point in $X/\Pi$ and for its corresponding subset of X.  For practical purposes use $\Pi$ for both $\Pi$ and $X/\Pi$.

Further consider
$$
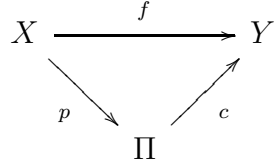
\xymatrix{X\ar[rr]^f\ar[dr]_p&&Y\\&\Pi\ar[ur]_c&}$$
where p is the projection, $p(x)=p([x])=[x]$, and $c\in C(\Pi,Y)$.  Set Fcn$(\Pi,Y)=\{h:\Pi\rightarrow Y\}$, continuity not assumed.    

\begin{lemma}
If Y is a magna then C(X,Y) is a magna.  If f and g are constant on an equivalence class of X then $f\cdot g$ is constant on that equivalence class.  This also holds for induced unitary operations on C(X,Y).
\end{lemma}

\begin{lemma}
 If the functions of C(X,Y) are constant on the equivalence classes of X and if Y is a magma there is an injective magma homorphism $\mathtt{G}:C(X,Y)\rightarrow$ Fcn$(\Pi,Y)$ which is a bijection if the equivalence classes are open in X.
\end{lemma}

\begin{proof}
  Define $\mathtt{G}(f)$ by $\mathtt{G}(f)([x])=f(x)$. That $\mathtt{G}$ is an injective algebraic homomorphism follows directly from its definition.
  
  Suppose that the [x] are open in X.  To see the surjectivity let $h:\Pi\rightarrow Y$ and define f by  
  $f(x)=h([x])$.  To see that f is continuous let $U\subset Y$ be open.  The set of points of $\Pi$ that G(f) map into U correspond to the union in X of the equivalence classes that f map into U.  As the equivalence classes are open so is this union.   
\end{proof}
 
Let $M(\Pi)$ be the free Magma generated by $\Pi$.  Consider $$
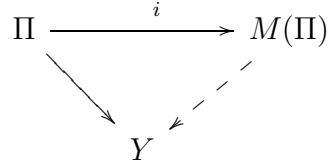
\xymatrix{\Pi\ar[rr]^i\ar[dr]&&M(\Pi)\ar@{-->}[dl]\\&Y&}$$ 
and the universal defining property of $M(\Pi)$.

\begin{lemma} 
  If $i^*:$Hom$(M(\Pi),Y)\rightarrow$ Fcn$(\Pi,Y)$ is the morphism induced by i then it is an isomorphism.
\end{lemma}

\begin{theorem}  
 If the functions of C(X,Y) are constant on equivalence classes and Y is a magma then $(i^*)^{-1}\circ\mathtt{G}: C(X,Y)\rightarrow$ Hom $(M(\Pi),Y)$ is an injective representation of C(X,Y) as a sub-magma of the magma of homorphisms of the free Magma into Y.
\end{theorem}

This result extends to C(X,Y) as a group or ring.

\begin{note}
Through out the paper we will expropriate the language of ring theory although most results require only one binary operation, $\cdot$.   A second binary operation, associativity, commutativity, or distributivity are not usually assumed being introduced in stages.  Modified definitions are given and  appropriate hypothesis are stated  as needed.  If Y has a unit for $\cdot$ (one sided or  two sided identity element), denoted 1,  then C(X,Y) has a unit of same type.  Denote this unit by Id or when convenient use I.
\end{note}

The topological information of X and algebraic information of C(X,Y) that we wish to relate is carried by $\Pi$ and $C(\Pi,Y)$. The projection p of X to $\Pi$ is clear. To compare C(X,Y) and $C(\Pi,Y)$ where X and Y are topological spaces and the functions of C(X,Y) are constant on equivalence classes, consider the following definitions.

\begin{defin} 
Define $G:C(X,Y)\rightarrow C(\Pi,Y)$ by $G(f)=f\circ p^{-1}$ and $H:C(\Pi,Y)\rightarrow C(X,Y)$ by $H(\hat{f})=\hat{f}\circ p$ where $\hat{f}\in C(\Pi,Y)$.
\end{defin}
 
Note: We are not assuming that equivalence classes are open in X.  Also remember that V is open in $\Pi$ if and only if $p^{-1}(V)$ is open in X.

\begin{lemma} 
 $H(\hat{f})$ is a continuous function.  As the functions of C(X,Y) are constant on equivalence classes, $G(f)=f\circ p^{-1}=\hat{f}$ is a continuous function. 
\end{lemma}

\begin{proof}
First $H(\hat{f})$ is a continuous function by its definition.  Second $G(f)=\hat{f}$ is a function as $f\circ p^{-1}(p(x))=f(x)$. To establish continuity let U is open in Y.   We need to show that 
$[G(f)]^{-1} (U)=(\hat{f})^{-1}(U)=(f\circ\ p^{-1})^{-1}(U)$ is open that is that $p^{-1}((f\circ\ p^{-1})^{-1}(U))$ is open.  This follows as $x\in p^{-1}((f\circ\ p^{-1})^{-1}(U))$ iff $p(x)\in (f\circ p^{-1})^{-1}(U)$ iff $(f\circ\ p^{-1})(p(x))\in U$ iff $f(p^{-1}(p(x)))\in U$ iff $f([x])\in U$ iff $f(x)\in U$ iff $x\in f^{-1}(U)$.
\end{proof} 

Henceforth assume that the functions of C(X,Y) are constant on the equivalence classes, [x].  This is a condition that clearly holds if the equivalence classes of X are connected sets and Y has the discrete topology.  Note $C(\mathbb{Z},\mathbb{Z}_{2})$ is extensively known.

\begin{lemma} 
For these sets of continuous functions both H and G are bijections which are inverses.
\end{lemma}

\begin{proof}
That they are injective is straight forward.  To see the surjection, for an f such that $G(f)=\hat{f}$ use $f=\hat{f}\circ p$ and for $\hat{g}$ such that $H(\hat{g})=g$ use $\hat{g}=g\circ p^{-1}$.
\end{proof}

For the next theorem assume that Y has a continuous binary operation as in Theorem 1.   Also introduce  operations in C(X,Y) and  $C(\Pi,Y)$ as in Theorem 1.  The operation need not be associative or commutative.

\begin{theorem} 
If Y has a continuous binary or unitary operation then C(X,Y) and $C(\Pi, Y)$ are corresponding isomorphic  structures with H and G inverse isomorphisms.
\end{theorem}

\begin{proof}
The proof is straight forward as these are function spaces keeping in mind that functions in C(X,Y) are constant on equivalent classes.  However the proof is tedious.
\end{proof}

The following observations through the Comment are about relations in general and will be used later.  Let X and Y be sets and F(X,Y) the functions from X to Y. 
   
\begin{defin} 
Let $J\subseteq F(X,Y)$ and define $x\cong_{J} y$ iff for all $f\in J$, $f(x)=f(y)$.  Denote the equivalence class by $[x]_{J}$ and the corresponding partition of X by $\Pi_{J}$.
Also define $x\cong y$ iff for all $f\in F(X,Y)$, $f(x)=f(y)$ that is when $J=F(X,Y)$.  Denote this class by $[x]$ and the correspond partition of X by $\Pi$.
\end{defin}

\begin{lemma} 
If $\varnothing\ne J\subseteq A\subseteq F(X,Y)$ then $[x] \subseteq [x]_{A}\subseteq [x]_{J}$.
\end{lemma}

In diffidence to our intention we now assume that Y has a unique element denoted 0.  For conveniencece call this element zero which could be any fixed element.  Later it will become a null element and then an algebraic zero. 

\begin{defin} 
Let $\varnothing\ne A\subseteq F(X,Y)$ and $b\in Y$.  Define the b set of A to be $V(A,b)=\{x\mid \forall f\in A, f(x)=b\}=\bigcap_{f\in A} f^{-1}(b)$.  If Y has the element 0, let $V(A)=V(A,0)=\bigcap_{f\in A}f^{-1}(0)$ be 
called the zero set of A.  Let $V(f)=V(\{f\})$.  Note in the case of $C(X,Y)$ if $\{b\}$ is closed in Y then $V(A,b)$ is closed in X.
\end{defin}

\begin{lemma} 
$\varnothing\ne J\subseteq A\subseteq F(X,Y)=F$ then $V(F,b)\subseteq V(A,b)\subseteq V(J,b)$.
\end{lemma}

\begin{defin} 
If $b\in Y$, $\varnothing\ne U\subseteq X$, and $\varnothing\ne J\subseteq F(X,Y)$  define $I(U,b)_{J}=\{f\mid  f\in J, f(U)=b\}$.  Let $I(U,b)=I(U,b)_{F}$ and if $Y$ has a zero element as above, denote $I(U,0)_{J}$ by $I(U)_{J}$ and $I(U)_{F}$ by $I(U)$.
\end{defin}

\begin{lemma}
Via definition $I(U,b)_{J}\subseteq J$ and so $I(U)_{J}\subseteq J$.
\end{lemma}

\begin{lemma}
$U\subseteq V(I(U,b)_{J},b)$ so $U\subseteq V(I(U,0))_{J},0)$ and therefore $U\subseteq V(I(U))$.
\end{lemma}

\begin{proof}
If $x\in U$ then $f\in I(U,b)_{J}$ implies that $f(x)=b$, that is $\forall f\in I(U,b)_{J}, f(x)=b$, so $x\in V(I(U,b)_{J},b)$.
\end{proof} 

\begin{lemma}
If $J\subseteq A$ then $I(U,b)_{J}\subseteq I(U,b)_{A}\subseteq I(U,b)_{F}$.
\end{lemma}

\begin{lemma} 
$J\subseteq I(V(J,b),b)$ and therefore $J\subseteq I(V(J))$.
\end{lemma}

\begin{lemma}
If $U_{1}\subseteq U_{2}\subseteq X$ then $I(X,b)_{J}\subseteq I(U_{2},b)_{J}\subseteq I(U_{1},b)_{J}$.
\end{lemma}

\begin{corollary}
If $\varnothing\ne J\subseteq A\subseteq F(X,Y)$ and $B\subseteq F(X,Y)$ then $I([x]_{J},b)_{B}\subseteq I([x]_{A},b)_{B}\subseteq I([x],b)_{B}$.
\end{corollary}

\begin{theorem} 
If for every f in J, $f(x)=b$ then $V(J,b)=[x]_{J}$
\end{theorem}

\begin{proof}
$V(J,b)=\bigcap_{f\in J}f^{-1}(b)=\{y\mid \forall f\in J, f(y)=f(x)=b\}=[x]_{J}$ where $f(x)=b$.
\end{proof}

\begin{corollary}
If for every $f\in J$, $f(x)=b$, and if $A\subseteq F(X,Y)$ then $I(V(J,b),c)_{A}=I([x]_{J},c)_{A}$.  In particular if $b=c=0$ then $I(V(J))=I([x]_{J})\subseteq I([x])$. 
\end{corollary}

Observation: Let $F(X,Y)=A_{1}\supseteq A_{2}\supseteq ...... \supseteq A_{n}$.  Then $[x]=[x]_{A_{1}}\subseteq [x]_{A_{2}}\subseteq ...... \subseteq [x]_{A{n}}$.  Thus if $\{x\}\neq \varnothing$ and $[x]_{A_{i}}=\{x\}$ then they are all equal for $j\leq i$.  Moreover $I([x]_{A_{1}},b)_{J}\supseteq I([x]_{A_{2}},b)_{J}\supseteq.....\supseteq I([x]_{A_{n}},b)_{J}$.  Finally recall $C(\Pi,Y)\xrightarrow{H} C(X,Y)$ is an  
isomorphism.  If functions are continuous on these equivalence classes of $A_{n}$ then by Theorem 3, $C(\Pi,Y)$ and all of the $C(\Pi_{A_{i}},Y)$ are algebraic isomorphic.  We could consider $(X/\Pi_{1})/\Pi_{2}$, but not at this time.

\begin{note}
Use $U^{c}$ for $X\setminus U$ for any set $ U\subset X$.  $U^{c}$ is the complement of $U$ in X.
\end{note}

Now proceed with X as a topological space.
   
\begin{defin} 
 The quasi-component of $x\in X$ is the intersection of all clopen (closed and open) subsets of X which contain x, [5], page 246.   The quasi-component of x is denoted $Q_x$.
\end{defin}

\begin{corollary}
If $y\in Q_{x}$ then $Q_{x}=Q_{y}$ otherwise there is a clopen set U containing y that does not contain x but then $U^{c}$ is clopen, $x\in U^{c}$ and $y\notin U^{c}$.  Quasi-components are disjoint and clopen sets are unions of quasi-components.
\end{corollary}

\begin{note}
A quasi component must be closed but need not be open.  For example consider $X=\{0,1,1/2,1/3,...,1/n,...\}\subseteq \mathbb{R}$ with the topology induced from $\mathbb{R}$.  Zero is a non-open quasi -component.
\end{note}

Hence forth unless otherwise indicated Y is a topological space such that for any two distinct points there is a clopen set containing one of the points but not the other.  Such a space is $T_{2}$ and totally separated.  Hence every point is a quasi-component, [10] page 32.  The rational numbers with the subspace topology inhereted from the reals is a prime example.  Note Y need not be zero dimensional that is it may not have a clopen set basis, for examples see [10].  

We are now interested in the equivalence relation $x\cong y$. 

\begin{note}
It is easy to create topological spaces, X, with any number of components and quasi-components with different topological properties.  Then choosing Y with algebraic properties, we can obtain algebraic structures on C(X,Y) which are different from Y and which reflect the quasi-components as situated in X.  Passing to $C(\Pi,Y)$ will facilitate proofs.
\end{note}

\begin{theorem} 
The elements, $[x]$,  of $\Pi$, the equivalence classes, are precisely the quasi-components of X.  That is $[x]=Q_x$ for all $x\in X$.
\end{theorem}

\begin{proof}
  If $Q_x\not\subseteq [x]$ then there is an $f\in C(X,Y)$ that is not constant on $Q_x$.  So there exist a $y\in Q_x$ such that $f(x)\neq f(y)$.  Let $U\subseteq Y$ be a clopen set about f(x) such that $f(y)\notin U$.  Then $f^{-1}(U)$ is a clopen set about x not containing y.  This is a contradiction to the definition of $Q_x$.  Therefore $Q_x\subseteq [x]$.

  On the other hand, let $y\in[x]$ and let V be a clopen set about x.  If there is only one V then V=X and there is nothing to prove.  Choose an a and b in Y such that $a\neq b$.  
The characteristic function $\chi(z)=\left\{\begin{array}{ll}b&\text{if } z\in V\\a&\text{if } z\notin V\end{array}\right.$ is continuous on $X$ with $\chi(x)=b$.  Since 
$y\in[x]$ it follows $\chi(y)=b$ and hence $y\in V$.  Thus $[x]$ is contained in any clopen set containing x and thus $[x]\subseteq Q_x$.
\end{proof}   
 
\begin{note}
 This is a new characterizaton of quasi-components. 
 \end{note}
  
\begin{corollary}
 Functions are constant on the quasi-components.
 \end{corollary}  
 
 \begin{corollary}
  If $U$ is clopen and $x\in X$ then either $U\supseteq[x]$ or $U\cap[x]=\phi$.
\end{corollary}

\begin{proof}
Either $[x]\subseteq U$ or $[x]\nsubseteq U$.  If $[x]\nsubseteq U$ then $x\notin U$.  So $x\in X\setminus U=U^{c}$ which is clopen.  Thus $U\cap [x]=\phi$.
\end{proof}

COMMENT:  If we assume Y has a binary operation with a zero 0 and unit 1 then quasi-components can be characterized using a slightly different equivalence relation as follows.  Consider the set of characteristic functions in C(X,Y) of the form $\chi_{u}$ which is 0 on U and 1 on $U^{c}$ where U is clopen.  Define $x\sim_{c} y$ if and only if for all $\chi_{u}\in C(X,Y)$, $\chi_{u}(x)=\chi_{u}(y)$.  Denoting the equivalence class of x by $[x]_{\chi}$ an easy adaptation of Theorem 5 shows that $Q_{x}=[x]_{\chi}$.

This establishes that the two equivalence classes are equal, that is $[x]=[x]_{\chi}$.  The set of ${\chi_{u}}$ is closed under multiplication in C(X,Y) but need not be closed under addition if Y is  ring.  Ultimately we are interested in C(X,Y) as a ring although multiplication will dominate through out most the development without addition or using associativity or commutativity for an operation.  

Convention:  When Y has a zero 0 and unit 1, $\chi_{u}$ will denote a continuous characteristic function as in the comment and $\chi_{u^{c}}$ analogously.  Let $\chi_{x}$ denote $\chi_{\{x\}}$.  Different characteristic functions when needed will be defined.

\begin{note}
As components are subsets of quasi-components, functions that are constant on a component are constant on its corresponding quasi-component.  Hence the interest in the quasi-components of X.
\end {note}

\begin{lemma} 
Any two quasi-components can be separated by a clopen set.
\end{lemma}

\begin{proof}
  Let [x] and [y] be two distinct quasi-components.  So there exist clopen sets U and V such that $U\supseteq [x]$, $V\supseteq [y]$, $x\notin V$ and $y\notin U$.  By the corollary $[x]\cap V=\phi$ and $[y]\cap U=\phi$.  So it follows that $[x]\subseteq U\cap V^{c}$ and $[y]\subseteq V\cap U^{c}$.  These are the desired 
  clopen sets.
\end{proof}

\begin{corollary}
Quasi-components are disjoint and every clopen set is a union of quasi-components.
\end{corollary}

Now returning to $C(\Pi,Y)$ and the paragraph before Lemma 1) consider
$$
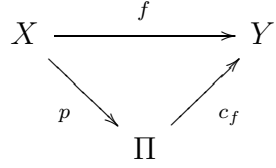
\xymatrix{X\ar[rr]^f\ar[dr]_p&&Y\\&\Pi\ar[ur]_{c_{f}}&}$$
where p is the projection and the function $c_{f}$ is defined in the natural way as follows.  

\begin{defin} 
Let $z_x=p([x])$ and set $c_f(z_x)=f(p^{-1}(z_x))$ and again $\Pi$ denotes $X/\Pi$.
\end{defin}

\begin{lemma}
$c_{f}$ is continuous.
\end{lemma}

\begin{proof}
If U is open in Y we need to show that $c_f^{-1}(U)$ is open in $\Pi$, that is that $p^{-1}((c_f^{-1}(U))$ is open in X.  First note that $f^{-1}(U)$ is open and is the union of all of the equivalence classes of the x such that $f(x)\in U$, that is $f([x])=f(x)\in U$.  The elements of $\Pi$ that $c_f=f\circ p^{-1}$ maps into U  correspond to the union of the quasi-components of X that f maps into U.  This union is $f^{-1}(U)$ and is $p^{-1}(c^{-1}(U))$ that is $f^{-1}(U)=p^{-1}(c^{-1}(U))$.  Therefore $c_f^{-1}(U)$ is open.  
\end {proof}

\begin{lemma}
If U is clopen in X then $U=p^{-1}p(U)$.
\end{lemma}

\begin{proof}
$x\in U$ iff $[x]\subset U$ iff $p(x)=p([x])\in p(U)$ iff $x\in [x]\subset p^{-1}p(U)$.
\end{proof} 

\begin{corollary}
If W is clopen in $\Pi$ then $p^{-1}(W)$ is clopen in X and conversely if U is clopen in X then p(U) clopen in $\Pi$.  Thus the intersection of clopen sets defining $[x]$ in X correspond to the intersection of clopen sets containing p(x)=z in $\Pi$.  That is the quasi-components are in bijective correspondence.
\end{corollary}

\begin{theorem} 
Each point of $\Pi$ is a quasi-component of $\Pi$.  Thus $\Pi$ is totally separated and hence $T_2$, [10], pg. 32.
\end{theorem}
  
 COMMENT:  If $X=\lbrace x\rbrace$ is a singleton, then C(X,Y), $C(\Pi,Y)$, and Y are in bijective correspondence and C(X,Y) and $C(\Pi,Y)$ are naturally algebraicly isomorphic to Y.
 However this occurs in other special cases.  Although $X=\lbrace x\rbrace$ has been excluded, $\Pi$ is a single point when X is connected.  Thus C(X,Y), $C(\Pi,Y)$, and Y would be in bijective correspondence.  For an example where C(X,Y) and Y are different consider $C(\mathbb{Z},\mathbb{Z}_2)$.  Here $\mathbb{Z}$ is topologically $\mathbb{Z}/\Pi$ and the cardinalities of $C(\mathbb{Z},\mathbb{Z}_2)$ and $\mathbb{Z}_2$ do not match.

\begin{note}
We will not consider $C(\Pi/\Pi, Y)$ as the intersection of clopen sets defining $[x]$ corresponds to the intersection of clopen sets defining $p([x])=z_{x}$ which correspond to the intersection of the clopen sets defining $\hat{z_{x}}=p(z_{x})\in C(\Pi/\Pi,Y)$.  Also we know that C(X,Y), $C(\Pi,Y)$, and $C(\Pi/\Pi,Y)$  are algebraically isomorphic.  Moreover $\Pi$ and $\Pi/\Pi$ are topologically isomorphic.
\end {note} 

Based on the development, this note, and Theorems 3, 5, and 6, it is convenient to replace $\Pi$ by a topological space Z whose points are quasi-component.  Denote this topology by $\mathcal{T}$ which is totally separated, [10] pg. 32 and thus $T_2$ and Urysohn.  Hence forth we consider C(Z,Y) and let the elements of C(Z,Y) be denoted by f.  $\newline$

\noindent  Example:  If $\mathbb{Z}_2$ is chosen for Z and $\mathbb{Z}_3$ for $Y$ then $\parallel \mathbb{Z}_2\parallel=2$, $\parallel \mathbb{Z}_3\parallel =3$, and $\parallel C(\mathbb{Z}_2,\mathbb{Z}_3)\parallel =9$, none of which are isomorphic as sets, topologies, or algebraic structures.

\begin{note}
For examples we could start with an X and use the topology $\mathcal{T}$ on $Z=X/\Pi=\Pi$ inherited from X by p. For convenience use Z for Z with this induced topology, $\mathcal{T}$.  Later we compare other topologys for Z denoted by $\mathcal{T}_{1}$ and $\mathcal{T}_\mathfrak{Z}$. 
\end{note}

\begin{note}
Despite the similarities of the topologies of Z and Y, C(Z,Y) can be very different from Y.  For examples take Y as $\mathbb{Z}_n$ and take Z as any discrete space of any cardinality greater than one.
\end{note}

We will now assume that Y has a continuous binary operation denoted by $\cdot$ with a left, right, or two sided null element, 0 (denoting the corresponding null function in C(Z,Y) by $\Theta$).  We can use either of the three null elements unless stated otherwise. For convenience use $0\cdot a=0$ for all a in Y.  When needed the unit for Y is denoted 1 and the unit for C(Z,Y) by Id or I.  Properties such as associativity or commutatively are still not assumed.  The 0 and 1 in Y and hence in C(Z,Y) are not unique unless $\cdot$ is commutative.  Also use the notation $(f\cdot g)(z)=f(z)\cdot g(z)$ in C(Z,Y).

\begin{corollary}
If Y has a unit then any characteristic function, $\chi$, zero on a set and one on its complement is idempotent that is $\chi\cdot \chi=\chi$ where $\chi$ need not be in C(Z,Y).  The characteristic function of the zero set is the Id and of Z is $\Theta$.
\end{corollary}

Recall that for set functions the next two lemmas hold.

\begin{lemma} 
If $B_{1}\subset B_{2}$ then $f^{-1}(B_{1})\subset f^{-1}(B_{2})$.  So if $b\in B$ then $f^{-1}(b)\subset f^{-1}(B)$.
\end{lemma}

\begin{lemma}
Suppose $f^{-1}(b)$ is not empty and $b= \cap B$ then $f^{-1}(b)=\cap f^{-1}(B)$.
\end{lemma}

\begin{proof}
The inverse of an intersection is the intersection of the inverses.  Thus $f^{-1}(b)=f^{-1}(\cap B)=\cap f^{-1}(B)$.
\end{proof}

From Definition 3 select:

\begin{defin} 
$B\subset Z$ is a zero set iff $B=V(f)=f^{-1}(0)$ for some $f\in C(Z,Y)$.
\end{defin}

\begin{lemma}
These sets are closed as $\{0\}$ is closed in $Y$.
\end{lemma}

\begin{lemma}
$V(f)\cup V(g)\subseteq V(f\cdot g)$
\end{lemma}

\begin{note} Observe that a clopen set U is the zero set of the  characteristic function., $\chi_{u}(z)=\left\{\begin{array}{ll}0&\text{if } z\in U\\a&\text{if } z\notin U\end{array}\right.$ where $a\neq 0$.
\end{note}

\begin{lemma}
A zero set, V(f), is an intersection of clopen sets.
\end{lemma}

\begin{proof}
In Y, $\{0\}$ is the intersection of all clopen sets containing it.  As the inverse of a clopen set is clopen and as the inverse of an intersection is the intersection of the inverses, $f^{-1}(0)$ is in an intersection of clopen sets.  The f image of a point in this intersection must be 0 otherwise it can be separated from 0 by clopen sets.
\end{proof}

\begin{theorem} 
Each point $z\in Z$ is the intersection of all zero sets containing it.
\end{theorem}

\begin{proof}
As each $z\in Z$ is a quasi-component and as the intersection of an intersection is just the intersection of the underlying sets, the intersection of the zero sets at z contain $\{z\}$ as a quasi-component.  If the two sets are not equal then there is a clopen set containing z which does not contain all of the intersection of zero sets.  But this clopen set is the zero set of a continuous characteristic function.
\end{proof}

\begin{corollary}
Each quasi-component $z\in Z$ is both the intersection of all zero sets and of all clopen sets containing it.
\end{corollary}

\begin{note}
However $\{z\}$ need not be a zero set.  For an example in the real line with the induced topology let 
$X= \{0, 1, 1/2, -----, 1/n, ----\}$.  Then $Z=X=\{z_0, z_1, z_2, -----\}$.  Here $\{ z_0 \}$ is a quasi-component which is closed but not open.  Let Y be a space such as $Z_2$ where $\{0\}$ is clopen.  Then for any continuous function f, $V(f)\neq \{z_0\}$
 \end{note}

Now from Definition 3 select:

\begin{defin} 
Let $S\subseteq C(Z,Y)$.  We say $V(S)=\{z\in Z: f(z)=0 \quad \mbox{for all}\quad f\in S\}=\bigcap_{f\in S}f^{-1}(0)=\bigcap_{f\in S}V(f)$ is the zero set of S.
\end{defin}

The V(S) are closed as each V(f) is closed.  Moreover every clopen set is the zero set of  continuous characteristic functions.  The characteristic functions zero on U and otherwise a will be denoted $\chi_{u,a}$.  When Y has a unit a will be 1 and write $\chi_{u}$.

\begin{corollary}
$V(\varnothing)=Z$.
\end{corollary}

\begin{corollary}
AS $\lVert Z\rVert \geq 2$, $V(C(Z,Y))=\varnothing$. 
\end{corollary}

\begin{lemma}
If $J\subseteq C(Z,Y)$ and $S\subseteq J$ then $V(J)\subseteq V(S)$.  If  Y has an additional associative operation denoted by addition and S generates J in the sense that if $f\in J$ then $f=\sum\limits_{i=1}^ng_{i}\cdot f_{i}$ where $f_i\in S$ and $g_i\in C(Z,Y)$, then V(S)=V(J).
\end{lemma}

However this additional operation is still not needed.

\begin{defin} 
A right zero-divisor is any element $y_1\in Y$ such that $y\cdot y_1=0$ for some $y\in Y$, $y\ne 0$.  For convenience assume $y_1\ne 0$.  An element $y\in Y$ is nilpotent iff $y^n=0$ for some integer n and idempotent iff $y\cdot y=y$.  Analogously define left and two sided zero divisors, nilpotent, and idempotent elements for Y an for C(Z,Y).  In the nilpotent case, $\cdot$ must be associative. 
\end{defin}

\begin{lemma}
If $Y$ has no right (left) divisors of zero then $V(f_1)\cup V(f_2)=V(f_{1}\cdot f_{2})$.
\end{lemma}

\begin{note}
This result suggest a natural topology for Z to be compared with $\mathcal{T}$.
\end{note}

\begin{theorem} 
If $Y$ has no right (left) divisors of zero then the zero sets of $Z$ in the sense of Definition 8 form a Zariski topology for $Z$ having the $V(S)$ as the topology of closed sets, [9] pg.s 6 and 7.  Denote this topology by  $\mathcal{T}_\mathfrak{Z}$.  When $\mathcal{T}_\mathfrak{Z}$ is present assume no divisors of zero.
\end{theorem}

\begin{lemma}
A set $U\subseteq Z$ open in $\mathcal{T}_\mathfrak{Z}$ is open in the topology $\mathcal{T}$ of $Z$.  So $\mathcal{T}_\mathfrak{Z} \subseteq \mathcal{T}$.
\end{lemma}

\begin{proof}
As $U$ is open in $\mathcal{T}_\mathfrak{Z}$ it is the complement of a zero sets of $\mathcal{T}$ and hence open in $\mathcal{T}$.  Alternately said the closed sets, $V(f)$, of $Z$ in $\mathcal{T}$ generate $\mathcal{T}_\mathfrak{Z}$.  Thus $\mathcal{T}_\mathfrak{Z}$ maybe  courser than $\mathcal{T}$.
\end{proof}
   
 \begin{note}
 As the clopen sets of $Z$ need not form a basis for the topology $\mathcal{T}$ of $Z$.  Let $\mathcal{T}_1$ be the topology which has these clopen sets as an open set basis, [7] pg 47, T11.  By duality as the clopen sets are also closed they form a closed basis for the $\mathcal{T}_1$ topology.  Clearly $\mathcal{T}_1\subseteq \mathcal{T}$.
\end{note}

\begin{lemma} 
$\mathcal{T}_1\subseteq \mathcal{T}_\mathfrak{Z}$.
\end{lemma}

\begin{proof}
Observe $U$ is a basic open set of in $\mathcal{T}_1$ if and only if $U$ is clopen in $\mathcal{T}$ if and only if $U^{c}$ is clopen in $\mathcal{T}$.  Thus $U^{c}$ is a zero set with respect to $\mathcal{T}$ which implies $U^{c}$ is a basic closed set in $\mathcal{T}_\mathfrak{Z}$ which implies $U=(U^{c})^{c}$ is open in $\mathcal{T}_\mathfrak{Z}$.  Therefore $\mathcal{T}_1\subseteq \mathcal{T}_\mathfrak{Z}$.
\end{proof} 

\begin{lemma}
$\mathcal{T}_\mathfrak{Z} \subseteq \mathcal{T}_1$.
\end{lemma}

\begin{proof}
Let $U\in \mathcal{T}_\mathfrak{Z}$ be an open set of $\mathcal{T}_\mathfrak{Z}$ and thus equal to $(V(S))^{c}$ for some $S\subset C(Z,Y)$.  Now $(V(S))^{c}=(\bigcap_{f\in S} f^{-1}(0))^{c}=\bigcup_{f\in S} (f^{-1}(0))^{c}$.  Recall $\{0 \}\in Y$ is a quasi component and $\{0 \} =\bigcap \{ K: K $ clopen and $  0\in K  \}$.  So $f^{-1}(0)=f^{-1}(\bigcap \{ K: K  $ clopen and $  0\in K  \})=\bigcap\{f^{-1}(K): K $ clopen and $ 0\in K \}$.  Thus $(f^{-1}(0))^{c}=(\bigcap \{f^{-1}(K): K $ clopen and $ 0\in K \})^{c}=\bigcup \{(f^{-1}(K))^{c} :
K $ clopen and $ 0\in K \}$.  Now as K is clopen $f^{-1}(K)$ is clopen implying $(f^{-1}(K))^{c}$ is clopen which implies that $U=\bigcup_{f\in S} (f^{-1}(0))^{c}$ is open in $\mathcal{T}_1$.  (For an easier conceptual proof use a sub-basis argument.)
\end{proof}

\begin{theorem} 
$\mathcal{T}_1=\mathcal{T}_\mathfrak{Z}\subseteq\mathcal{T}$.
\end{theorem}

\begin{corollary}
 The quasi-components of $\mathcal{T}_1$ and $\mathcal{T}_\mathfrak{Z}$ are the same.
 \end{corollary}
 
 \begin{corollary}
 If open sets of $\mathcal{T}$ contain clopen sets then $\mathcal{T}=\mathcal{T}_{1}=\mathcal{T}_{\mathfrak{Z}}$ as then $\mathcal{T}\subseteq \mathcal{T}_{1}$.
 \end{corollary}
 
 \begin{lemma} 
 A set $U$ is clopen in $\mathcal{T}$ if and only if it is clopen in $\mathcal{T}_1$.
 \end{lemma}
 
 \begin{proof}
 Let $U$ and hence $U^{C}$ be clopen in $\mathcal{T}$.  Then they are both open in $\mathcal{T}_1$ and being complements are both closed, therefore clopen in $\mathcal{T}_1$. 
 
 Let $U$ be clopen in $\mathcal{T}_1$.  By Theoem 9 it is clopen in $\mathcal{T}$.
 \end{proof}
 
 \begin{theorem} 
 The quasi-components of $\mathcal{T}$, $\mathcal{T}_1$,  and $\mathcal{T}_\mathfrak{Z}$ are the same, namely $\{z\}$ where $z\in Z$.  Moreover their sets of clopen quasi-components are the same.
 \end{theorem}
 
 As $\mathcal{T}$ is The original topology for Z we denote (Z,$\mathcal{T}$) by Z and continue to use C(Z,Y) for C((Z,$\mathcal{T}$),Y).
 
 \begin{lemma} 
 If Y has a clopen base for its topology then $C(Z,Y)= C((Z,\mathcal{T}_1),Y)$.    
 \end{lemma}
  
 \begin{proof}
 Let $f\in C(Z,Y)$.  If $V$ is open in $Y$ it is a union of clopen sets of $Y$.  Therefore $f^{-1}(V)$ is a union of clopen sets of $\mathcal{T}$.  Thus it is open in $\mathcal{T}_1$ as it is a union of open sets of $\mathcal{T}_1$.  Therefore $f\in C((Z,\mathcal{T}_{1}),Y)$. 
 
 Let $f\in C((Z,\mathcal{T}_1),Y)$ and let $V$ be open in $Y$.  Then $f^{-1}(V)$ is open in $(Z,\mathcal{T}_1)$.  By Theorem 9 it is open in $\mathcal{T}$ and thus $f\in C(Z,Y)$.
 \end{proof}
 
 \begin{lemma}
 Moreover C(Z,Y), $C((Z,\mathcal{T}_1),Y)$, and $C((Z,\mathcal{T}_\mathfrak{Z}),Y)$ are the same as sets.
 \end{lemma} 
 
\begin{theorem} 
If Y has an open base topology then $C(Z,Y)=C((Z,\mathcal{T}_1),Y)=C((Z,\mathcal{T}_\mathfrak{Z}),Y)$ are algebraically isomorphic with the same continuous functions, cf. Theorem 1.
\end{theorem}

\begin{note}
Although $\mathcal{T}$ for want of evidence in some cases could be stronger than $\mathcal{T}_{1}$ in what follows $\mathcal{T}$ is used.  It may be that in any case $C((Z,Y)$ and $C((Z,\mathcal{T}_{1}),Y)$ are algebraic isomorphic and that $\mathcal{T}_{1}$ has for an open basis the clopen sets of $\mathcal{T}$.
\end{note}

For convenience we will drop the modifiers right or left as most of the proofs use characteristic functions $\chi$. 
 
\begin{theorem} 
Although Y may not have a divisor of zero if $Z$ has two or more elements then $C(Z,Y)$ has a zero-divisor.  If Y has nilpotent elements so does C(Z,Y).
\end{theorem}

\begin{proof}
Let $U$ be a clopen set in $Z$ containing one point but not another.  Then choose $a\in Y$, $a\neq 0$ and define $\psi_{u,a}(z)=\left\{\begin{array}{ll}a&\text{if } z\in U\\0&\text{if } z\not\in U\end{array}\right.$ and $\chi_{u,a}(z)=\left\{\begin{array}{ll}0&\text{if } z\in U \\a&\text{if } z\not\in U\end{array}\right.$.  Note $\psi_{u,a}$ and $\chi_{u,a}$ are continuous, neither are the zero function, and $\Theta=\chi_{u,a}\cdot\psi_{u,a}$.  Therefore $C(Z,Y)$ has a zero divisor.

If $a\in Y$ is nilpotent then $\chi_{u,a}$ is nilpotent.
\end{proof}

\begin{corollary}
For every $a\neq 0, a\in Y$ and for every clopen set $U\subseteq Z$, $\chi_{u,a}$ is a distinct divisor of zero. This says that the cardinality of the set of divisors of zero exceeds that of Z, $\mathcal{T}$, and Y.  Mutatis Mutandis for nilpotent elements when $\cdot$ in Y is associative. 
\end{corollary}
   
 \begin{corollary}
 If $C(Z,Y)$ has no zero divisors then $Z$ consist of a single point and therefore $C(Z,Y)$ and $Y$ are algebraically isomorphic.  
 \end{corollary}
 
 \begin{corollary}
 If $f\in C(Z,Y)$ is nilpotent then for every $z\in Z$, f(z) is nilpotent as is $\chi_{\varnothing, f(z)}$.
 \end{corollary} 
 
 \begin{note}
$\mathbb{Z}_3$ has no zero divisors but $C(\{z_1, z_2\}, \mathbb{Z}_3)$ does.
\end{note}

 Adapting Definition 4 for C(Z,Y) were Y has the annihilator, 0:
 
\begin{defin} 
For $U\subset Z$ let $I(U)$ denote those functions that annihilates $U$ that is $I(U)=\{f: f(z)=0 \quad \mbox{for all}\quad z\in U\}$.  Use I(z) for I(\{z\}).
\end{defin}

\begin{corollary}
If $U_{1}\subseteq U_{2}$, then $I(U_{2})\subseteq I(U_{1})$.
\end{corollary}

\begin{corollary}
$\{ z\}=V(I(z))$.
\end{corollary}

\begin{proof}
If $z_1\in V(I(z))$ and $z_1\neq z$, then they can be separated by clopen sets.
\end{proof}

\begin{corollary}
If $1\in Y$ and if U is clopen then $I(U)=(\chi_{u})$ where $f\in (\chi_{u})$ gives $f=f\cdot \chi_{u}$.
\end{corollary}

 \begin{corollary}
 If $U_1$ is clopen, $U_{2}\neq \varnothing$, and $U_1 \subsetneqq U_2$ then $I(U_1)\neq I(U_2)$
 \end{corollary}
  
  \begin{corollary}
 If $z_1\neq z_2$ then $I(z_1)\neq I(z_2)$ and neither is a subset of the other.
  \end{corollary}
  
  \begin{proof}
  This follows as $z_{1}$ and $z_{2}$ can be separated by disjoint clopen sets.
  \end{proof}
    
 \begin{corollary}
 If $z\in U$ but $\{z\}\neq U$ then $I(z)\neq I(U)$ although $I(U)\subset  I(z)$.
 \end{corollary}
 
\begin{proof}
 If $\{z\}\ne U$ there exists $y\in U,  y\ne z$.  So there is a clopen set $U_{z}$ such that $y\notin U_{z}$.  Therefore there is a characteristic function zero on $U_{z}$ and thus on $z$ but not on $y$, so not on $U$.
 \end{proof}

\begin{lemma} 
 If $U_1$ is clopen and $U_1^{c}\cap U_2\neq \emptyset$ where $U_2$ need not be clopen,  then $I(U_1)\neq I(U_2)$.
 \end{lemma}
 
 \begin{proof}
As there is a characteristic function $\chi_{u_1}$ which does not vanish on $U_1^{C}$ and therefore not on $U_2$, $I(U_1)\neq I(U_2)$.
 \end{proof}
  
\begin{defin} 
Call a subset $I$ of $C(Z,Y)$ which contains the zero function, $\Theta$, and which is closed with respect to the operation $\cdot$, a right (left) multiplicative ideal if $f\in C(Z,Y)$ and $g\in I$ then $f\cdot g\in I,(g\cdot f\in I)$.   A right (left) ideal $I \ne C(Z,Y)$ is prime if $f\cdot g\in I$ implies  $f\in I$ or $g\in I$.  Keep in mind that $\chi_{u}\cdot \chi_{u^c}=\chi_{u^c}\cdot\chi_u=\Theta$.  If I is generated by f, $I=(f)=\{g\cdot f\vert g\in C(Z,Y)\}$, it is called principal. 
\end{defin}

Recall as the algebraic properties and definitions for C(Z,Y) are to be defined from those of Y we assume the weakest properties of interest when not specified.  The I(z) although ideals in the weak multiplicative sense will be called ideals.  However the development is such that the results will apply after additional operations for Y are assumed and finally ring ideals are obained.

\begin{corollary}
By the third Corollary to Definition 10, if $1\in Y$ and  if U is clopen  then I(U)=($\chi_{u}$) the principal ideal generated by $\chi_{u}$.
\end{corollary}

\begin{corollary}
$(\Theta)$ is not prime as for any clopen U, $\chi_{u,a}\cdot \chi_{u^c,a}=\Theta$.
\end{corollary}

\begin{corollary}
If I is a prime ideal then there is an $f\in I$ such that $V(f)\neq \varnothing$.
\end{corollary}

\begin{proof}
As the cardinality of Z in at least 2, let U be a clopen set containing one point but not another.  Then $\Theta= \chi_{u,a}\cdot \chi_{u^{c},a}$ and thus $\chi_{u,a}$ or $\chi_{u^{c},a}$ is in I neither of which has an empty zero set.
\end{proof} 

\begin{lemma} 
For any U, $I(U)$ is an ideal.  If Y has no divisors of zero then $I(z)$ is prime.
\end{lemma}

\begin{proof}
To see that I(z) is prime let $f\cdot g\in I(z)$.  So $f(z)\cdot g(z)=0$ and hence as Y has no divisors of zero f(z)=0 or g(z)=0. This says f or g is in I(z).
\end{proof} 

\noindent Example:  The multiplication of $\mathbb{Z}_4$ has divisors of zero.  Choosing the discrete topology on $\mathbb{Z}_4$ and on $Z=\{z_1, z_2\}$, a set of two elements, the ideals $I(z_1)$ and $I(z_2)$ are not prime ideas of $C(Z, \mathbb{Z}_4)$. 

\begin{theorem} 
If $I$ is an ideal of functions vanishing on two or more points, ie. $\lVert V(I)\rVert \geq 2$  then $I$ is not prime.  
\end{theorem}

\begin{proof}
In $Z$ let $z_1$ and $z_2$ be be two such distinct points of V(I).  By Lemma 13, $z_1$ and $z_2$ can be separated by two clopen sets $U_1$ and $U_2$ respectively with characteristic functions $\chi_{u_1^{c}}$ and $\chi_{u_2^{c}}$ such that $\chi_{u_1^{c}}\cdot \chi_{u_2^{c}}=\Theta\in I$ but neither zero on $\{z_1,z_2\}$.
\end{proof}

\begin{corollary}
If $I(U)$ is a prime ideal and $U\ne \varnothing$ then $U$ is  a singleton.  Also observe that $I(\varnothing)=C(Z,Y)$.
\end{corollary}

\begin{proof}
The claim for $I(\varnothing)$ is from first order logic.
\end{proof}

\begin{corollary}
If $J\subseteq I$, both ideals, and $\lVert V(I)\rVert \geq 2$ then $J$ is not prime.
\end{corollary}

\begin{proof}
As $V(I)\subseteq V(J)$, $2\leq \lVert V(I)\rVert\leq \lVert V(J)\rVert$.
\end{proof}

\begin{theorem} 
If Y has no divisors of zero then $I(U)$ is a prime ideal if and only if U is a singleton that is there exist a $z$ such that $U=\{z\}$.
\end{theorem}

\begin{proof}
Use the Corollary to Theorem 13 and Lemma 30.
\end {proof}

\begin{corollary}
If Y has no divisors of zero then the intersection of the prime ideals is trivial.
\end{corollary}

\begin{lemma} 
If Y has no divisors of zero and its binary operation is associative, then C(Z,Y) has no nontrivial nilpotent
elements.
\end{lemma}

\begin{proof}
Suppose $f\in C(Z,Y)$, $f\neq \Theta$, and $f^{n}= \theta$ for some n.  Thus if $f(z)=a\neq 0$, then $f^{n}(z)=a^{n}=a^{n-1}\cdot a=\theta (z)=0$.
\end{proof}

\begin{corollary}
If Y is a ring and has no divisors of zero then the prime (nill) radical of C(Z,Y) is trivial and the prime spectrum of C(Z,Y) is disconnected.
\end{corollary}

\begin{proof}
The intersection of all I(z) is $\{\Theta\}$, the only nilpotent  element.  If U is clopen note $\chi_{u}\cdot \chi_{u}=\chi_{u}$ and use [1] pg 13, exercise 22 (i) and (iii) for the disconnectivity.
\end{proof}

\noindent  Now enlarge the investigation not only to obtain a better grasp of the ideals but to examine possible directions of interest.

\begin{lemma} 
If I is a proper ideal then: \newline 1) If $U \and\ U_1$ are clopen, $U\subseteq U_1$, and $\chi_{u,a}\in I$ then $\chi_{u_1,a}\in I$. \newline 2)  For every U clopen and I prime, either $\chi_{u,a}\in I$ or $\chi_{u^{c},a}\in I$. Thus every prime ideal has for every clopen set U, a divisor of zero , either $\chi_{u,a}$ or $\chi_{u^{c},a}$.  These are idempotents when $1\in Y$ and a=1.  
\end{lemma}

\noindent Clearly if U is clopen and $z\in u$ then $\chi_{u,1}=\chi_{u}\in I(z)$.  \newline  More generally:

\begin{lemma}
For an ideal I if $f\in I$ then for any clopen U, $f\cdot \chi_{u,a}\in I$ and $f\cdot \chi_{{u}^{c},a}\in I$.
\end{lemma}

\begin{theorem} 
If $1\in Y$ then every nontrivial ideal I has a function f such that $f\neq \Theta$, U=V(f) is clopen, and $\varnothing\neq V(f)\neq Z$.
\end{theorem}

\begin{proof}
Choose $g\in I$ such that $g\neq \Theta$.  Thus $V(g)\neq Z$.  If $V(g)=\varnothing$ let $U\neq Z$ be clopen.  Then $f=g \cdot \chi_{u}\in I$ and U=V(f).  \newline  If $V(g)\neq \varnothing$ as $g\neq \Theta$ for $y\in Range(g)\setminus \{0\}$, choose W clopen such that $0\in W$ and $y\notin W$.  Then for $U=g^{-1}(W)$, $f=g\cdot \chi_{u}\in I$, $V(g)\subseteq V(f)\neq Z$, and V(f)=U is clopen.
\end{proof}

\begin{note}
Two operations in C(Z,Y) are used in Lemma 35, Theorems 16 and 18, and used from Definition 12 on although not always needed.   For convenience call $f\cdot g$ multiplication and when f+g is introduced and call it addition.  As before let $\Theta$ be zero and I or Id be the unit.  When the + operation is used the previous results are not effected by this additional operation.  For example if Y is a ring then I(z) is a ring ideal.  For expediency a ring was used in the Corollary to Lemma 31 although the results is more general.  In what follows when addition is present distribution may not be needed.
\end{note}

\begin{corollary}
If Y has addition $\chi_{u}+\chi_{u^{c}}$=Id.  Note we need only that $1\in Y$ and that 1+0=1.
\end{corollary}

\begin{lemma} 
If $J\subset I(z)$ and $V(J)\neq V(I(z))=\{ z\}$ then J is not prime.
\end{lemma}

\begin{proof}
This is a Corollary to Theorem 13.
\end{proof}

\begin{lemma} 
If C(Z,Y) has addition with a unit, if \{z\} is clopen, and if $J\subsetneqq I(z)\neq C(Z,Y)$ is an ideal then J is not prime.
\end{lemma}

\begin{proof}
As {z} is clopen I(z)=($\chi_{z}$).  If J is prime so $\Theta = \chi_{z^{c}} \cdot \chi_{z}\in J$ then $\chi_{z}$ or $\chi_{z^{c}}$ is in J.  If $\chi_{z}\in J$ then $J\supseteq (\chi_{z})=I(z)$ contrary to $J\subsetneqq I(z)$.  If $\chi_{z^{c}}\in J$ then $Id=\chi_{z} +\chi_{z^{c}} \in I(z)$ contrary to $I(z)\neq C(Z,Y)$.
\end{proof}

\begin{theorem} 
If Y has no divisors of zero, has addition with a unit, and if \{z\} is clopen, then I(z) is a minimal prime ideal.
\end{theorem}

\begin{lemma}
From the third Corollary to Definition 10, if $1\in Y$,and, if $\{z\}$ is open, that is isolated, then $I(z)=(\chi_{z})$ is principal and $\chi_{z}$ is an idempotent and a divisor of zero. 
\end{lemma}

\begin{corollary}
If I(z)=(f) is principal and $\{z\}$ is open then $(\chi_{z})=(f)$.
\end{corollary}

\begin{corollary}
If I(z) is not $(\chi_{z})$ then the quai-component \{z\} is not open.
\end{corollary}

\begin{lemma} 
If $f\neq \Theta$ and V(f) is clopen then f is a divisor of zero.
\end{lemma}

\begin{proof}
If V(f) is clopen then $V^{c}(f)$ is clopen. Thus $\chi_{v^{c}(f)}\cdot f=\Theta$.
\end{proof}

\begin{lemma} 
If $V(f)\supseteq \{z\}$, then $I(z)\supseteq (f)$.
\end{lemma}

\begin{lemma}
 If V(f)$\supseteq V(\chi_{u})=U$ then (f)$\subseteq ( \chi_{ u})$.
\end{lemma}

\begin{lemma} 
In general V(f)=V((f)). 
\end{lemma}

\begin{proof}
As $f\in (f)$ then $V((f))\subseteq V(f)$.  Suppose $z\in V(f)$ and recall for any $h\in (f)$, $h=g\cdot f$ for some g.  So h(z)=0.  Thus for all $h\in (f)$, h(z)=0 which says $z\in \bigcap_{h\in (f)}h^{-1}(0)=V((f))$.
\end{proof}

\begin{corollary}
If I(z)=(f) then $\{z \}$=V(f).
\end{corollary}

\begin{proof}
As $V(h\cdot f)\supseteq V(h)\bigcup V(f)$, the second Corollary to Definition 10 says, $\{z \}=V(I(z))=V((f))=V(f)$.
\end{proof}

\begin{note}
Recall if $1\in Y$ a  prime ideal has a divisor of zero that is an idempotent for every clopen subset $U\subsetneq Z$,  either $\chi_u$ or $\chi_{u^{c}}$.  Now consider divisors of zero in general.  Lemmas 41 and 42 are independent of the topologies of Z and Y and holds if ,Y is a magma with no "divisors of zero."
\end{note}

\begin{lemma} 
If f is a divisor of zero and $g\cdot f=\Theta$, $g\neq \Theta$, and if Y has no divisors of zero then 
 $V^{c}(f)\bigcap V^{c}(g)=\varnothing$.
 \end{lemma}
 
 \begin{proof}
 If $z\in V^{c}(f)\bigcap V^{c}(g)$ then $f(z)\neq 0\neq g(z)$ but $g\cdot f=\Theta$ says $g(z)\cdot f(z)=0$ contrary to Y has no divisors of zero.
 \end{proof}
 
 \begin{lemma}
 If f is a divisor of zero and Y and g as above then $Z=V^{c}(f)\bigcup V^{c}(g)$, disjoint.
 \end{lemma}
 
 \begin{proof}
 For $z\in Z, g(z)\cdot f(z)=\Theta(z)=0$.  As Y has no divisors of zero at least one but by the previous Lemma not both f(z) or g(z) is zero.
 \end{proof}
 
 \begin{theorem} 
 If f, g, and Y are as above then $Z=V^{c}(f)\bigcup V^{c}(g)=V(f)\bigcup V(g)$  are disjoint unions and V(f),V(g), $V^{c}(f)$ and $V^{c}(g)$ are all clopen.
 \end{theorem}
 
 \begin{corollary}
 If Y has no divisors of zero then f is a divisor of zero iff $V(f)\neq Z$ is clopen in which case $(f)\subseteq (\chi_{v(f)})$.
 \end{corollary}
 
 \begin{note}
 When Y has no divisor of zero and $V(f)=\{z \}$ then f is a divisor of zero iff z is isolated in which case I(z)=($\chi_{z})\supseteq (f)$.  This is a consequence of the Corollary and Lemma 39.
 \end{note}
 
\begin{corollary}
If Y has no divisor of zero and $\{z \}$ is not open then: \newline 1) f is not a divisor of zero or $V(f)\neq \{z \}$ in which case $I(z)\neq (f)$.  \newline 2) if f is a divisor of zero then $V(f)\neq \{z \}$ in which case $I(z)\neq (f)$.  \newline 3) if $V(f)=\{z \}$ then f is not a divisor of zero.
\end{corollary}
  
\begin{theorem} 
If $I_1$ and $I_2$ are ideals, $I_{2}\neq C(Z,Y)$, $1\in Y$, Y has addition, $I_1$ prime, $I_{1}\subseteq I_{2}$, and U clopen then $\chi_{U}\in I_{1}$ if and only if $\chi_{U}\in I_{2}$.
\end{theorem}

\begin{proof}
Note $\chi_{\varnothing}=Id$.  So suffices to consider $\chi_{U}\in I_{2}$.  If $\chi_{U}\notin I_{1}$ then $\chi_{U^C}\in I_{1}$.  So $\chi_{U}+ \chi_{U^C}=Id\in I_{2}$ contrary to $I_{2}\neq C(Z,Y)$.
\end{proof}

\begin{corollary}
If $I_{1}\bigcap I_{2}$ is prime then $\chi_{U}\in I_{1}$ if and only if $(\chi_{U})\subseteq I_{1}\bigcap I_{2}$.
\end{corollary}

\noindent   It is important to note if Y has a unit and the two operations then even without distributivity of the two operations, $\chi_{u}+\chi_{u^{c}}=Id$ .  

\begin{note}
The sequent of results  suggest that there are binary set operations on the set of prime ideals, the set of clopen sets, and the set of characteristic functions.  These operations then satisfy further properties.  First recall that if I is a prime ideal then $\chi_{u,a}$ or $\chi_{u^c,a}$ is in I.  However although stated in terms of primality, it is not needed in many of the following cases yielding results for ideals in general.  The results from Definition 12 through the Summary Note following Lemma 55 are related to possible convergence structures in C(Z,Y) and may not be of algebraic interest.  However some of the set function techniques are used later.
\end{note} $\newline$

\noindent  Now assume $1\in Y$.  If and when two operations are used and it is easy to see if additional hypotheses are needed, they will not be stated except when they add to the flow of the narrative.
   
\begin{defin} 
Let $\Phi$ be the set of all ideals and let $\mathcal{P}$ be the set of all prime ideals together with $(\theta)$.  Define for each clopen U, \newline  $\Phi_{u}=\{I\in \Phi\vert \chi_{u}\in I\}$,       \newline $\mathcal{P}_{u}=\{ I\in \mathcal{P}\vert \chi_{u}\in I\}=\lbrace I\in \mathcal{P}\vert (\chi_{u})\subseteq I\rbrace$, and thus $\newline
\mathcal{P}_{u^{c}}=\lbrace I\in \mathcal{P}\vert \chi_{u^{c}}\in I\rbrace =\lbrace I\in \mathcal{P}\vert (\chi_{u^{c}})\subseteq I\rbrace$.
\end{defin}

\begin{lemma}
1)  $\chi_{z}=\Theta$ and $\chi_{z}\in \bigcap \mathcal{P}$,  \newline 2)  $\chi_{\varnothing}=Id$ as $\chi_{\varnothing}(z)$ is 1 if $z\notin \varnothing$ and 0 otherwise. \newline 3)  $\Phi_{z}=\Phi$ and $\mathcal{P}_{z}=\mathcal{P}$ as $\chi_{z}=\Theta\in I$ for all I.  \newline  4)  $\Phi_{\varnothing}=\mathcal{P}_{\varnothing}=\lbrace C(Z,Y)\rbrace$ as $\chi_{\varnothing}=Id$.  \newline 5)  For any clopen U, $C(Z,Y)\in \mathcal{P}_{u}$.
\end{lemma}

\begin{lemma} 
If $I_1$ and $I_2$ are ideals, $I_{2}\neq C(Z,Y)$,  and $I_{1}\subseteq I_{2}$ then $I_{1}\in \mathcal{P}_{u}$ if and only if $I_{2}\in \Phi_{u}$.  Thus if $I_2$ is also prime, $I_{1}\in \mathcal{P}_{u}$ if and only if $I_{2}\in \mathcal{P}_{u}$.
\end{lemma}

\begin{proof}
Theorem 18.
\end{proof}

\begin{lemma}
For each non empty clopen set U, $U\neq Z$, $\mathcal{P}=\mathcal{P}_{u}\bigcup \mathcal{P}_{u^{c}}$ and $\mathcal{P}_{u}\bigcap \mathcal{P}_{u^c}=\varnothing$.  
\end{lemma} 

\begin{proof}
At lease one but not both $\chi_{u}$ and $\chi_{u^c}$ can belong to I.
\end{proof}

\begin{lemma} 
If U and W are clopen then $\mathcal{P}_{u}\subseteq \mathcal{P}_{u\bigcup w}\bigcap \mathcal{P}_{u\bigcup w^{c}}$.  (also for $\Phi_U$)
\end{lemma}

\begin{proof}  
Let $I\in \mathcal{{P}}_{u}$.  It follows that $\chi_{w}\cdot \chi_{u}=\chi_{u\bigcup{w}}$ and $\chi_{w^{c}}\cdot \chi_{u}$ are in I.
\end{proof}

\begin{corollary}
If $U_{1}\subseteq U_{2}$ are clopen then $\mathcal{P}_{u_{1}}\subseteq \mathcal{P}_{u_{2}}\bigcap \mathcal{P}_{u_{1}\bigcup u_{2}^{c}}$ and  $\mathcal{P}_{u_{1}}\subseteq \mathcal{P}_{u_{2}}$ as $\chi_{u_{2}}=\chi_{u_{2}}\cdot \chi_{u_{1}}$.  So if $U_{2}^{c}\subseteq U_{1}^{c}$ then $\mathcal{P}_{u_{2}^{c}}\subseteq \mathcal{P}_{u_{1}^{c}}$.  (also for $\Phi_u$).
\end{corollary}

\begin{corollary}
Let $U_1$ and $U_2$ be clopen: \newline 1)  Recall if $U_{1}\subseteq U_2$ then $\mathcal{P}_{u_1}\subseteq \mathcal{P}_{u_2}$.  \newline  2) $\mathcal{P}_{u_{1}\bigcup u_{2}}\supseteq \mathcal{P}_{u_1}\bigcup \mathcal{P}_{u_2}$.  \newline  3)  $\mathcal{P}_{u_{1}\bigcap u_{2}}\subseteq \mathcal{P}_{u_{1}}\bigcap \mathcal{P}_{u_{2}}\subseteq \mathcal{P}_{u_{1}\bigcup u_{2}}$.  \newline  4)  If $U_{1}\bigcap U_{2}=\varnothing$ then $\mathcal{P}_{u_{1}}\bigcap \mathcal{P}_{u_{2}}=\{C(X,Y)\}$.  \newline  (corollary also holds for $\Phi_u$).
\end{corollary} 

\begin{proof}
For 3) use the preceding corollary and 2) of this corollary.  \newline For 4), if $U_{1}\bigcap U_{2}=\varnothing$ then $U_{1}^{c}\supseteq U_2$.  Let $I\in \mathcal{P}_{u_{1}}\bigcap \mathcal{P}_{u_{2}}$ so that $\chi_{u_{1}}$ and $\chi_{u_{2}}$ belong to I.  Then $\chi_{u_{1}^{c}}=\chi_{u_{1}^{c}}\cdot \chi_{u_{2}}\in I$.  Consequently both $\chi_{u_{1}}$ and $\chi_{u_{1}^{c}}$ belong to I.  Thus I=C(Z,Y) and $\mathcal{P}_{u_{1}}\bigcap \mathcal{P}_{u_{2}}=\{ C(Z,Y)\}$.
\end{proof} 

\begin{defin} 
Let $\mathcal{U}=\{U\vert \medspace U is \medspace clopen\}$.  \newline Define for $I\in \mathcal{P}$, \newline  $\mathcal{U}_{I}=\{U\vert  \chi_{u}\in I\}=\{ U\vert (\chi_{u})\subseteq I\}$ and \newline  $\mathcal{U}_{I}^{c}=\{U\vert \chi_{u^{c}}\in I\}$.
\end{defin}

\begin{lemma} 
1)  $\mathcal{U}_{(\Theta)}=\{ U\vert \chi_{u}\in (\Theta)\}=\{Z\}$ as $\chi_{Z}=\Theta$. \newline 2)  $\mathcal{U}_{C(Z,Y)}=\mathcal{U}_{C(Z,Y)}^{c}=\mathcal{U}$. \newline 3) $\mathcal{U}_{(\theta)}^{c}=\{\varnothing\}$. \newline 4) $\mathcal{U}_{(\theta)}\bigcup \mathcal{U}_{(\theta)}^{c}=\{\varnothing ,Z\}\neq \mathcal{U}$.  \newline 5) If $I_1$ and $I_2$ are in $\mathcal{P}$ and $I_{1}\subseteq I_{2}$ then $\mathcal{U}_{I_{1}}\subseteq \mathcal{U}_{I_{2}}$ and if  $I_{2}\neq C(Z,Y)$ then $\mathcal{U}_{I_{1}}=\mathcal{U}_{I_{2}}$.
\end{lemma}

\begin{lemma}
$I\in \mathcal{P}_{u}$ if and only if $U\in \mathcal{U}_{I}$.
\end{lemma}

\begin{corollary}
$(\Theta)\in \mathcal{P}_{Z}$ if and only if $Z\in \mathcal{U}_{(\Theta)}$.
\end{corollary} 

\begin{lemma} 
If $I\in \mathcal{P}\setminus \{C(Z,Y)\}$ then $\mathcal{U}=\mathcal{U}_{I} \bigcup \mathcal{U}_{I}^{c}$ and $\mathcal{U}_{I} \bigcap \mathcal{U}_{I}^{c}=\varnothing$.  
\end{lemma}

\begin{lemma} 
If $I_{1}\bigcap I_2$ is prime, then $\mathcal{U}_{I_{1}\bigcap I_{2}}=\mathcal{U}_{I_{1}}\bigcap \mathcal{U}_{I_{2}}$.
\end{lemma}

\begin{proof}
$U\in \mathcal{U}_{I_{1}}\bigcap \mathcal{U}_{I_{2}}$ iff $U\in \mathcal{U}_{I_{1}}$ and $U\in \mathcal{U}_{I_{2}}$  iff $\chi_{u}\in I_{1}$ and $\chi_{u}\in I_{2}$ iff $\chi_{u}\in I_{1}\bigcap I_{2}$ iff $I_{1}\bigcap I_{2}\in \mathcal{P}_{u}$ iff $U\in \mathcal{U}_{I_{1}\bigcap I_{2}}$. 
\end{proof} 

\begin{lemma} 
If $I_{1}\bigcup I_{2}$ is a prime ideal, then $\mathcal{U}_{I_{1}}\bigcup \mathcal{U}_{I_{2}}=\mathcal{U}_{I_{1}\bigcup I_{2}}$.
\end{lemma}

\begin{proof}
($U\in \mathcal{U}_{I_{1}}\bigcup \mathcal{U}_{I_{2}}$) if and only (if $U\in \mathcal{U}_{I_{1}}$ or $U\in \mathcal{U}_{I_{2}}$) if and only if ($\chi_{u}\in I_{1}$ or $\chi_{u}\in I_{2}$) iff  ($\chi_{u}\in I_{1}\bigcup I_{2}$) iff ($U\in \mathcal{U}_{I_{1}\bigcup I_{2}}$).
\end{proof}

\begin{defin} 
Let $\mathcal{X}=\{ \chi_{u}\vert \medspace U \medspace is \medspace clopen \}$. \newline For $I\in \mathcal{P}$ define $\mathcal{X}_{I}=\{ \chi_{u} \vert \medspace \chi_{u}\in I \}\subseteq I$ and \newline $\mathcal{X}_{I}^{c}=\{ \chi_{u} \vert \medspace \chi_{u^{c}}\in I \}$.
\end{defin}

\begin{lemma}
1) $\mathcal{X}_{(\Theta)}=\{\chi_{u} \vert \medspace \chi_{u}\in (\Theta)\}=\{\Theta\}$ and $\mathcal{X}_{C(Z,Y)}=\mathcal{X}$.  \newline 2) If $I_1$and $I_2$ are in $\mathcal{P}$ and $I_{1}\subseteq I_{2}$ then $\mathcal{X}_{I_{1}}\subseteq \mathcal{X}_{I_{2}}$ and if $I_{2}\neq C(Z,Y)$ then $\mathcal{X}_{I_{1}}=\mathcal{X}_{I_{2}}$. \newline 3) $\chi_{u}\in \mathcal{X}_{I}$ if and only if $U\in \mathcal{U}_{I}$ if and only if $I\in \mathcal{P}_{u}$.
\end{lemma}

\begin{corollary}
$\chi_{z}=\Theta\in \mathcal{X}_{(\Theta)}$ iff $Z\in \mathcal{U}_{(\Theta)}$ iff $(\Theta)\in \mathcal{P}_{Z}=\mathcal{P}$ and $\chi_{\varnothing}\in \mathcal{X}_{C(Z,Y)}$ if and only if $\varnothing \in \mathcal{U}_{C(Z,Y)}$ if and only if $C(Z,Y)\in \mathcal{P}_{\varnothing}$.
\end{corollary}

\begin{lemma} 
$\mathcal{X}=\mathcal{X}_{I}\bigcup \mathcal{X}_{I}^{c}$ and $\mathcal{X}_{I} \bigcap \mathcal{X}_{I}^{c}=\varnothing$.
\end{lemma}

\begin{lemma} 
If $I_1$, $I_2$, and $I_{1}\bigcap I_{2}$ are in $\mathcal{P}$ (or $I_1$ and $I_2$ are in $\Phi$ necessary changes made)  then $\mathcal{X}_{I_{1}}\bigcap \mathcal{X}_{I_{2}}=\mathcal{X}_{I_{2}\bigcap I_{2}}$.
\end{lemma}

\begin{lemma}
If $I_1$, $I_2$, and $I_{1}\bigcup I_{2}$ are in $\mathcal{P}$ (or $\Phi$), then $\mathcal{X}_{I_{1}}\bigcup \mathcal{X}_{I_{2}}=\mathcal{X}_{I_{1}\bigcup I_{2}}$.
\end{lemma}

\begin{note}
In summary:  Let $U_1$ and $U_2$ be clopen.
\newline A) $\mathcal{P}_{\varnothing}=\{C(Z,Y) \}$; 
 $\mathcal{P}_{Z}=\mathcal{P}$; 
 If $U_{1}\subseteq U_{2}$, 
 then $\mathcal{P}_{u_{1}}\subseteq \mathcal{P}_{u_{2}}$; 
$\mathcal{P}_{u_{1}\bigcap u_{2}}\subseteq \mathcal{P}_{u_{1}}\bigcap \mathcal{P}_{u_{2}}$;  and $\mathcal{P}_{u_{1}}\bigcup \mathcal{P}_{u_{2}}\subseteq \mathcal{P}_{u_{1}\bigcup u_{2}}$.
\newline B)  $\mathcal{U}_{(\Theta)}=\{Z\}$; 
$\mathcal{U}_{C(Z,Y)}=\mathcal{U}$;
 if $I_1$ and $I_2$ are in $\mathcal{P}$ and $I_{1}\subseteq I_{2}$, then $\mathcal{U}_{I_{1}}\subseteq \mathcal{U}_{I_{2}}$ and if  $I_{2}\neq C(Z,Y)$, then $\mathcal{U}_{I_{1}}=\mathcal{U}_{I_{2}}$; 
 if $I_{1}\bigcap I_{2}$ is  prime then $\mathcal{U}_{I_{1}\bigcap I_{2}}=\mathcal{U}_{I_{1}}\bigcap \mathcal{U}_{I_{2}}$; and  if  $I_{1}\bigcup I_{2}$ is a prime ideal then $\mathcal{U}_{I_{1}}\bigcup \mathcal{U}_{I_{2}}=\mathcal{U}_{I_{1}\bigcup I_{2}}$.
\newline C)  $\mathcal{X}_{(\Theta)}=\{\Theta\}$; $\mathcal{X}_{C(Z,Y)}=\mathcal{X}$; if $I_1$ and $I_2$ are in $\mathcal{P}$ and $I_{1}\subseteq I_{2}$ then, $\mathcal{X}_{I_{1}}\subseteq \mathcal{X}_{I_{2}}$ and if $I_{2}\neq C(Z,Y)$ then $\mathcal{X}_{I_{1}}=\mathcal{X}_{I_{2}}$; if $I_{1}\bigcap I_{2}$ is prime then, $\mathcal{X}_{I_{1}}\bigcap \mathcal{X}_{I_{2}}=\mathcal{X}_{I_{1}\bigcap I_{2}}$; and if $I_{1}\bigcup I_{2}$ s a prime ideal then $\mathcal{X}_{I_{1}}\bigcup \mathcal{X}_{I_{2}}=\mathcal{X}_{I_{1}\bigcup I_{2}}$.
\end{note}

\begin{note}
 By Zorn's lemma every prime ideal of C(X,Y) with the one operation, $\cdot$, contains a minimal prime ideal with respect to $\cdot$, follow the proof of Lemma 2, in [3] pg 73.
\end{note}

\begin{theorem} 
If J is prime and $V(J)\neq \varnothing$ then there is a unique z such that $V(J)=\{z\}$ and $J\subseteq I(z)$.
\end{theorem}

\begin{proof}
If $V(J)\neq \varnothing$ then $\parallel V(J)\parallel \geq 1$.  If $\parallel V(J)\parallel >1$ then J is not prime.  So $\parallel V(J)\parallel=1$ that is there exist an unique z such that $V(J)=\{z\}$ and thus for all $f\in J$, $f(z)=0$.
\end{proof} 

\begin{corollary}
If $\{z\}=V(J)$ is open  then $J=I(z)=(\chi_z)$.  That is a prime ideal  J with a non empty zero set must be of the form I(z) where $\{z\}= V(J)$ even if Y has a divisor of zero.
\end{corollary}

\begin{proof}
Note $\chi_{z}$ or $\chi_{z^{c}}$ is in J.  So $\chi_{z}$ must be in J.  If $f\in I(z)$ then $\{z\}\subseteq V(f)$.  Therfore $f=f\cdot \chi_{z}\in J$.
\end{proof}

\begin{theorem} 
If $J\subsetneqq I$ where I is an ideal and $J$ is a prime prime ideal, and $I\neq C(Z,Y)$, then for every clopen U, one but not both of $\chi_{u}$ and $\chi_{u^{c}}$ is in I. 
\end{theorem}

\begin{proof}
As $J\subset I$ and as one but not both is in J then this is true for I by Theorem 18.
\end{proof}

\begin{note}
Let $Z$ is a discrete space of three points and $Y$  be $\mathbb{Z}_2$.  Then for example the ideals of functions vanishing on two points are not prime.
 \end{note}
  
\begin{lemma} 
If $Y$ is a division ring and $f(z)\neq 0$ for all $x\in Z$ then $g(z)=\frac{1}{f(z)}\in C(Z,Y)$.
\end{lemma}

\begin{proof}
Let $U$ be open in $Y$ and $0\notin U$ and let $U^{-1}=\{1/u: u\in U\}$. Then $g^{-1}(U) = \{z :f(z)\in U^{-1}\} = f^{-1}(U^{-1})$ which is open in $Z$.
\end{proof}

\begin{lemma}
If $Y$ is a division ring, $I$ a proper function ideal of $C(Z,Y)$ then for each $f\in I$ there exist $z\in Z$ such that $f(z)=0$ otherwise $f$ is invertible and thus $I=(1)$.
\end{lemma}

\begin{corollary}
If $Y$ is a division ring and $I=(f)=\lbrace h\cdot f\rbrace$ for all $h\in C(Z,Y)$ which is a proper ideal then as there exist an $z$ such that $f(z)= 0$,  $(f)\subset I(z)$.
\end{corollary}

\noindent Recall: 1)   Every non empty subset of $Z$ is a set of quasi-components.  \newline 2)  As $H:C(X,Y)\rightarrow C(\Pi, Y)$ is an algebraic isomorphism, $H$ maps ideals to ideals of the same algebraic kind and vice versa.

\noindent Moreover:
\begin{lemma} 
If $I(x)$ is an ideal in $C(X,Y)$ corresponding to $[x]$ and $z=[x]$ then $H(I(x))=I(z)$.
\end{lemma}

\begin{proof}
Let $f\in I(x)$, that is $f(x)=0$.  Now as $f\in C(X,Y)$, $H(f)(z)=f\circ p^{-1}(z)=f([x])=f(x)=0$.  Thus $H(f)\in I(z)$.  That is $H(I(x))\subseteq I(z)$.
Now suppose $\hat{f}\in I(z)$ that is $\hat{f}(z)=0$.  As $\hat{f}\in C(\Pi, Y)$, $G(\hat{f})\in C(X,Y)$ and $G(\hat{f})(x)=(\hat{f}\circ p)(x)=\hat{f}(z)=0$, that is $G(\hat{f})\in I(x)$.  Then $H(G(\hat{f}))=\hat{f}$, that is $\hat{f}\in H(I(x))$ and so $I(z)\subseteq H(I(x))$.
\end{proof}

It is now appropriate to consider set and algebraic observations about $\chi_{u}, \mathcal{X}, \mathcal{X}_{I}, \and\ \mathfrak{X}$ where $\mathfrak{X}$ is yet to be defined.

\begin{lemma} 
The following are already given or evident where U, $U_1$, and $U_2$ are clopen and $I\subset C(Z,Y)$ is an ideal.  When needed 1+0=0+1=1.
\newline 1)   $\chi_{u}\cdot \chi_{u}=\chi_{u}$, that is $\chi_{u}$ is idempotent and if Y has addition, then $\chi_{u}+\chi_{u^{c}}=Id$.  
\newline 2)   $\chi_{u_{1}}\cdot \chi_{u_{2}}=\chi_{(u_{1}\cup u_{2})}$ and if $\chi_{u_{1}}$ or $\chi_{u_{2}}$ is in I then $\chi_{(u_{1}\cup u_{2})}=\chi_{u_{1}}\cdot \chi_{u_{2}}$ is in I. 
\newline 3)  If Y has addition and 1+1=0, then $\chi_{u_{1}}+\chi_{u_{2}}=\chi_{((u_{1}\cap u_{2})\cup (u_{1}\cup u_{2})^{C})}$ and $\chi_{u}+\chi_{u}=\chi_{Z}=\Theta$.  
\newline 4)  If $U_{1}\subset U_{2}$ then $\chi_{u_{2}}=\chi_{u_{2}}\cdot \chi_{u_{1}}$. 
\newline  5) If $U_{1}\neq U_{2}$ then $\chi_{u_ {1}}\neq \chi_{u_{2}}$.  
\newline 6)  $U_{1}\bigcap U_{2}=V(\chi_{(u_{1}\cap u_{2})})$.  
\newline 7)  $U_{1}\bigcup U_{2}=V(\chi_{u_{1}})\bigcup V(\chi_{u_{2}})=V(\chi_{(u_{1}\cup u_{2})})=V(\chi_{u_{1}}\cdot \chi_{u_{2}})$.  
\newline 8)  If I is a prime ideal then for all clopen U, $\chi_{u}$ or $\chi_{u^{c}}$ is in I.  
\newline  9)  If $U\subseteq U_1$, and $\chi_{u}\in I$ then $\chi_{u_{1}}\in I$. 
\newline 10)  If $I\neq C(Z,Y)$ is a prime ideal , $U_{1}\bigcap U_{2}=\varnothing$, and $\chi_{u_{1}}\in I$ then $\chi_{u_{2}}\notin I$. 
\newline 11)  If in Y, 1+1=0, if I is closed under addition (or if $\chi_{(u_{1}\cap u_{2})}\in I$), and if $\chi_{u_{1}}$ and $\chi_{u_{2}}$ are in I then then $\chi_{u_{1}}+\chi_{u_{2}}=\chi_{((u_{1}\cap u_{2})\cup (u_{1}\cup u_{2})^{C})}\in I$. 
\newline 12) If $\varnothing \neq U\neq Z$ then $(\chi_{u})\bigcap (\chi_{u^{c}})=(\theta)$ and $(\chi_{u})+(\chi_{u^{c}})=C(Z,Y)$.  
\newline 13) $Id=\chi_{\varnothing}$ and for any ideal I, $\Theta=\chi_{Z}\in I$ and thus $\mathcal{X}_{I}\neq \varnothing$.
\newline  14) $\mathfrak{X}_{I}$ is multiplicatively closed. 
\newline  15) If $\chi_{u}\in \mathcal{X}_{I}$ then $\chi_{w} \cdot \chi_{u}=\chi_{w\cup u}\in \mathcal{X}_{I}$. 
\newline  16)  From 11, if Y and I are closed under addition (or if $\chi_{u\cap w}\in I$), 1+1=0, and $\chi_{u}$ and $ \chi_{w}$ are in $ \mathcal{X}_{I}$ then $\chi_{u}+\chi_{w}\in \mathcal{X}_{I}$.
\newline  17)  If $\chi_{u}\cdot \chi_{w}=\chi_{u\cup w}\in \mathcal{X}_{I}$, then $\chi_{u}$ or $\chi_{w}$ is in $\mathcal{X}_{I}$.
\newline 18)  If the operations are distributive, $\chi_{u}\in I$, and $\chi_{u}+\chi_{w}=\Theta$ then $\chi_{w}\in I$.  
\newline 19) If $\chi_{v}\cdot \chi_{u}$ and  $\chi_{v}+\chi_{u}$ are in I and 1+1=0 then $V\bigcap U\neq \varnothing$.
\end{lemma}

\begin{proof}
For 10).  We can use $U_{2}\subseteq U_{1}^{c}$ which implies that $\chi_{({u_{1}}^{c})}=\chi_{({u_{1}}^{c})}\cdot \chi_{u_{2}}$.  If $\chi_{u_{2}}\in I$ then $\chi_{({u_{1}}^{c})}\in I$ and therefore  $\chi_{u_{1}}+\chi_{({u_{1}}^{c})}=$Id$\in I$ which is a contradiction.  
\newline For 18).  So $\chi_{u}\cdot \chi_{u^{c}}+\chi_{w}\cdot \chi_{u^{c}}=\chi_{u^{c}}\cdot \Theta$ which says $\theta+\chi_{w}\cdot \chi_{u^{c}}=\theta$ implying that $\chi_{w}\in I$.  
\newline For 19).   Note $\chi_{(v\cup u)}=\chi_{v}\cdot \chi_{u}\in I$ implies that $\chi_{(v\cup u)^{c}}\notin I$.  Also observe as 1+1=0, $\chi_{v}+\chi_{u}=\chi_{((v\cap u)\cup (v\cup u))^{c})}=\chi_{(v\cap u)}\cdot \chi_{(v\cup u)^{c}}\in I$.  Together these imply that  $\chi_{(v\cap u)}\in I$.  Now if $V\bigcap U=\varnothing$ then $\chi_{(v\cap u)}=Id\in I$. 
\end{proof}

\begin{theorem} 
If $\cdot $ in Y is closed, associative, and commutative then $\mathcal{X}=(\mathcal{X},\Theta,+,\cdot)$ has the properties that $\cdot$ is closed, associative, and commutative and is such that every element is idempotent.  Moreover by 3), if Y has addition that is closed, associative, commutative, and 1+1=0  then + is closed, associative, commutative, and $\chi_{U}+\chi_{U}=\Theta$, that is $-\chi_{u}$ is $\chi_{u}$.  If the operations in Y are distributive then $\mathcal{X}$ is distributive.   In this case $\mathcal{X}$ is a subring of C(Z,Y) isomorphic to $C(Z,\mathbb{Z}_{2})$.
\end{theorem}

\begin{theorem} 
By 14), $\mathcal{X}_{I}$ is  multiplicative closed, by 15) it is an ideal, and by 17) it is prime.  Now by 16), if 1+1=0 in Y, if Y is a ring, and if I is an ideal then $\mathcal{X}_{I}$ is a prime ring idea in $\mathcal{X}$ and a subring of C(Z,Y).
\end{theorem}

\noindent Now consider the set of $\mathcal{X}_{I}$.
 
\begin{defin} 
Set $\mathfrak{X}=\{\mathcal{X}_{I} \vert I\in \mathcal{P}\}$ and set $\mathfrak{O}=\{\Theta\}=\{\chi_{Z}\}=\mathcal{X}_{(\Theta)}$.
\end{defin}

\begin{lemma} 
When appropriate, define addition and multiplication in $\mathfrak{X}$ in the natural way.  
\newline  1) If $I_{1}\subseteq I_{2}$ then $\mathcal{X}_{I_{1}}\subseteq \mathcal{X}_{I_{2}}$.  
\newline  2)  If $I_{1}\bigcap I_{2}\in \mathcal{P}$ then $\mathcal{X}_{I_{1}}\bigcap \mathcal{X}_{I_{2}}=\mathcal{X}_{{I_{1}}\bigcap {I_{2}}}$.  
\newline  3)   If $I_{I}\bigcup I_{2}\in \mathcal{P}$ then $\mathcal{X}_{I_{1}}\bigcup \mathcal{X}_{I_{2}}=\mathcal{X}_{{I_{1}}\bigcup I_{2}}$.  
\newline 4)   $\mathcal{X}_{I}+\mathfrak{O}=\mathcal{X}_{I}+\mathcal{X}_{(\Theta)}=\mathcal{X}_{I}$. \newline 5)  $\mathcal{X}_{I}\cdot \mathfrak{O}=\mathfrak{O}$.  
\newline 6)  If 1+1=0 in Y, Y is a ring, and $I_1$ and $I_2$ are ideals then by Theorem 22, $\mathcal{X}_{I_{1}}+\mathcal{X}_{I_{2}}$ and $\mathcal{X}_{I_{1}}\cdot \mathcal{X}_{I_{2}}$ are ideals in $\mathcal{X}$ but need not 
be in $\mathfrak{X}$.  
\newline  7)  If 1+1=0 in Y, Y a ring, and $I_1$ and $I_2$ are ideals then $\mathcal{X}_{I_{1}}$ and $\mathcal{X}_{I_{2}}$ are prime ideals, $\mathcal{X}_{I_{1}}+\mathcal{X}_{I_{2}}$ is in $\mathcal{X}$, and in $\mathfrak{X}$ whenever $I_{1}+I_{2}$ is prime .
\newline 8)  If 1+1=0 in Y then $\mathcal{X}_{I_{1}}+\mathcal{X}_{I_{2}}\subseteq \mathcal{X}_{I_{1}+I_{2}}$ as in general $\mathcal{X}_{I}\subseteq I$.
\newline  9)  If 1+1=0 in Y, Y is a ring, and $I_1$ and $I_2$ are ideals  then $\mathcal{X}_{I_{1}}\cdot \mathcal{X}_{I_{2}}$ is in $\mathcal{X}$ and in $\mathfrak{X}$ whenever $I_{1}\cap I_{2}$ is prime.
 \end{lemma}

\begin{proof}  
For 7).   From Theorem 22, $\mathcal{X}_{I_{1}}$ and $\mathcal{X}_{I_{2}}$ are ideals of $\mathcal{X}$ and thus their sum is an ideal of $\mathcal{X}$.  To show this ideal is in $\mathfrak{X}$ let $\chi_{u_{1}}\in \mathcal{X}_{I_{1}}$ and $\chi_{u_{2}}\in \mathcal{X}_{I_{2}}$.  By Lemma 59, No. 11) $\chi_{u_{1}}+\chi_{u_{2}}=\chi_{u_{3}}\in I_{1}+I_{2}$.
\noindent  For 9).  First $\mathcal{X}_{I_{1}}\cdot \mathcal{X}_{I_{2}}$ is an ideal.  Suppose $\chi_{u}\cdot \chi_{w}\in \mathcal{X}_{I_{1}}\cdot \mathcal{X}_{I_{2}}$.  Then $\chi_{u}\cdot \chi_{w}$ is a finite sum of terms of the form $\chi_{u_{i}}\cdot \chi_{w_{i}}$ where $\chi_{u_{i}}\in I_{1}$ and $\chi_{w_{i}}\in I_{2}$.  Thus $\chi_{u_{i}}\cdot \chi_{w_{i}}\in I_{1}\cap I_{2}$ and consequently $\chi_{u}\cdot \chi_{w}\in I_{1}\cap I_{2}$.
\end{proof}

\begin{note}
As $\mathfrak{X}$ does not in general have algebraic properties like $\mathcal{X}$ when 1+1=0 in Y and because $\chi_{u}+\chi_{u^{c}}=Id$ and  $\chi_{u}\cdot \chi_{u^{c}}=\Theta$ in $\mathcal{X}$, the following definition is of interest when $I_{1}+I_{2}$ and  $I_{1}\cdot I_{2}$ are present.
\end{note}

\begin{defin} 
In a structure R like a ring, possibly weaker, $f\in R$ has has a complement $f^{c}\in R$ iff $f+f^{c}=Id$ and $f\cdot f^{c}=\Theta$.  R has complements iff every element has a complement.
\end{defin}

\begin{corollary}
The only ideal of R that contains an element and its complement is R.  Thus if I is a proper ideal then f and $f^{c}$ are not both in I.
\end{corollary}

\begin{corollary}
If f has a complement then it and its complement are idempotent.  Thus if R has complements it is Boolean.
\end{corollary}

\begin{corollary}
If $I_{1}+I_{2}=R$ where $I_{1} \and\  I_{2}$ are proper ideals and $f\in R$ has a complement,  then f is in one of the ideals and its complement is in the other.
\end{corollary}

\begin{corollary}
If f has an additive inverse and a complement then its complement is unique (addition associative and commutative).
\end{corollary}

\begin{theorem} 
If R has complements, $I_{1}$ is a prime ideal, $I_{2}$ an ideal, and  $I_{1}+I_{2}\neq R$, then $I_{1}+I_{2}$ is a prime ideal.
\end{theorem}

\begin{proof}
As $I_{1}+I_{2}$ is an ideal let $f=f_{1}\cdot f_{2}\in I_{1}+I_{2}$.  As $I_{1}$ is prime and $f\cdot f^{c}=\Theta$, f or $f^{c}$ is in $I_{1}$.  If $f^{c}\in I_{1}$ then $f^{c}+f=(f^{c}+\Theta)+f=Id\in I_{1}+I_{2}=R$.  Thus  $f=f_{1}\cdot f_{2}\in I_{1}$ and consequently $f_1$ or $f_2$ is in $I_{1}$ from which it follows that $f_{1}+\Theta$ or $f_{2}+\Theta$ is in $I_{1}+I_{2}$ 
 \end{proof}
  
\begin{lemma} 
If  the two operations are associative and distributive, if $f_{1} \and\  f_{2}$ are multiplicative commuting idempotents,  and if their product has an additive inverse then $(f_{1}\cdot f_{2})^{c}=Id-f_{1}\cdot f_{2}$
\end{lemma}

\noindent  Note that for characteristic functions $\chi_{u}^{c}=\chi_{u^{c}}$.

\begin{corollary}
If $\chi_{u_{1}} \ and\ \chi_{u_{2}}$ are in C(Z,Y) or $\mathcal{X}$ and Y is a ring then $(\chi_{u_{1}}\cdot \chi_{u_{2}})^{c}=(\chi_{u_{1}\cup u_{2}})^{c}=Id-\chi_{u_{1}\cup u_{2}}=\chi_{(u_{1}\cup u_{2})^{c}}$.  Consequently if 1+1=0 then $\chi_{u_{1}}+\chi_{u_{2}}=\chi_{u_{1}\cap u_{2}}-\chi_{u_{1}\cup u_{2}}$.
\end{corollary}

\begin{lemma} 
In $C(Z,\mathbb{Z}_{2})$ or $\mathcal{X}$ if $I_{1}  \and\  I_{2}$ prime ideals and $\chi_{u}\cdot \chi_{w}\in I_{1}\cdot I_{2}$ then \newline $\chi_{u}^{c}\cdot \chi_{w}^{c}-(\chi_{u}\cdot \chi_{w})^{c}=\chi_{u\cup w}-\chi_{u\cap w}$.
\end{lemma}

\begin{proof}
Observe $(\chi_{u}\cdot \chi_{w})^{c}=Id-\chi_{u\cup w}$ and $\chi_{u}^{c}\cdot \chi_{w}^{c}=\chi_{u^{c}}\cdot \chi_{w^{c}}=\chi_{u^{c}\cup w^{c}}=Id-(\chi_{u^{c}\cup w^{c}})^{c}=Id-\chi_{u\cap w}$ and subtract the first from the second.
\end{proof}

\begin{note}
If any point of $Y$ is open then all points of $Y$ are open and hence $Y$ has the discrete topology.  We will use the term $\{0\}$ is open to indicate that Y is discrete.  Also in what follows Y need not be a ring that is $\cdot$ will usually suffice.  Thus again the ideals need not be ring ideals.
\end{note}

\noindent Here it is convenient to state a stronger form of Theorem 16.

\begin{theorem} 
If $\{z\}$ is clopen in Z, J a prime ideal, and $J\subseteq I(z)$ then $J=I(z)$.  Thus if $\{z\}$ is open in Z, I(z) is a minimal prime as it contains a prime ideal.
\end{theorem}

\begin{proof}
As $\chi_{z}\cdot \chi_{{z}^{c}}=0$, $\chi_{z}\in J$.  Thus if $f\in I(z)$ then $f=f\cdot \chi_{z}\in J$.
\end{proof}

\begin{theorem} 
If $\{0\}$ is clopen in $Y$, J is a prime ideal, and $J\subseteq I(z)$ then $J=I(z)$.  Thus if $\{0\}$ is open in Y, the I(z) are the min-max prime ideal. 
\end{theorem}

\begin{proof}
For $f\in I(z)$ choose $U=f^{-1}(0)$.  Then $\chi_{U}\in J$ and consequently $f=f\cdot \chi_{U}\in J$.
\end{proof} 

\begin{lemma} 
If $\{z\}$ is clopen then $I(z)=(\chi_{z})$ is principal and thus $\chi_{U}\in I(z)$ for every clopen $U$ containing $z$.  
\end{lemma}

\begin{proof}
Use Lemma 38 and Lemma 59, No.4.  Note that $\chi_{z}$ is continuous independent of whether $\{0\}\subseteq Y$ is open or not.
\end{proof}

\begin{note}
If $\chi$ is a characteristic function and $W\subset Y$ is clopen such that $0\in W$ and $1\notin W$ then $\chi^{_-1}(0)=\chi^{-1}(W)$ is clopen although $\{0\}\subset Z$ need not be open. Likewise $\chi^{-1}(1)$ is clopen.
\end{note}

\begin{lemma} 
If $(\chi_{u})$ is neither $(\Theta)$ or C(Z,Y) then neither is $(\chi_{u^{c}})$ and If $\cdot$ is commutative and associative then $(\chi_{u})\cap (\chi_{u^{c}})=(\Theta)$.  
\end{lemma}

\begin{proof}
If $(\chi_{u^{c}})=C(Z,Y)=(\chi_{\varnothing})=(\chi_{Z^{c}})$ then $(\chi_{u})=(\chi_{Z})=(\Theta)$.  \newline  If $(\chi_{u^{c}})=(\Theta)=(\chi_{Z})$ then $(\chi_{u})=(\chi_{Z^{c}})=(\chi_{\varnothing})=(Id)=C(Z,Y)$.  \newline  Suppose $\chi_{w}\in (\chi_{u})\cap (\chi_{u^{c}})$ and thus $\chi_{w}=\chi_{w_{1}}\cdot \chi_{u}=\chi_{w_{2}}\cdot \chi_{u^{c}}$.  As $\chi_{w}=\chi_{w}\cdot \chi_{w}$ form the composite product to show $\chi_{w}=\Theta$.
\end{proof}

\begin{theorem} 
If $\cdot$ is commutative and associative, if $J$ be a prime ideal, and if $J\neq C(Z,Y)$ or $(\Theta)$, then let $\chi_{1}\in J$ and set $U_{1}=\chi_{1}^{-1}(0)\neq Z$ which is clopen.  If $U$ is clopen and $U\cap U_{1}\ne \varnothing$ then $\chi_{u}\in J$.
\end{theorem}

\begin{proof}
Let $\chi_{u_{1}}=\chi_{1}\in (\chi_{u_{1}})$.  Cases: (a) $U_{1}\subseteq U$, (b) $U\subseteq U_{1}$, and (c) $U_{1}\bigcap U$ is in both. \newline  (a) Note, $\chi_{u}=\chi_{u}\cdot \chi_{u_{1}}\in (\chi_{u_{1}})\subseteq J$.  \newline  (b) If $U\subseteq U_{1}$ then $U_{1}^{c}\subseteq U^{c}$ which implies $\chi_{u^{c}}=\chi_{u^{c}}\cdot \chi_{u_{1}^{c}}\in(\chi_{u_{1}^{c}})$.  Thus $\chi_{u}\notin (\chi_{u_{1}^{c}})$ and thus $\chi_{u}\in (\chi_{u_{1}})\subseteq J$. 
 \newline  (c)  As $U\bigcap U_{1}$ is clopen and as $U\bigcap U_{1}\subseteq U_{1}$ by (b) $\chi_{u\bigcap u_{1}}\in J$.  Now as $U\bigcap U_{1}\subseteq U$ we have from (a), $\chi_{u}\in J$.
\end{proof}

\begin{lemma} 
Assuming that $\cdot$ is commutative and associative let  $J$ be a prime ideal, and $\chi_{u_{1}}\in J$.  If $U$ is clopen and if $z\in U\cap U_{1}$ then $\chi_{u}\in J\cap I(z)$.
\end{lemma}

\begin{theorem} 
If $\cdot$ is commutative and associative, if J is a prime ideal,  and if $\{0\} \subseteq Y$ is open then there exist a $\{z\}$ such that $I(z)\subseteq J$.  Note $\{z\}\subset Z$ need not be open.
\end{theorem} 

\begin{proof}
Let $U$ be clopen.  As $\chi_{u}$ or $\chi_{u^{c}}$ is in $J$ with out loss of generality call it $\chi_{u_{1}}$.  Choose $z\in U_1$ and $f\in I(z)$.  Let $U_{f}=f^{-1}(0)=V(f)$ which is clopen.  Then $z\in U_{1}\cap U_{f}$.  Thus by Theorem 26, $\chi_{f}\in J$ and therefore $f=f\cdot \chi_{f}\in J$.  Consequently $I(z)\subseteq J$.
\end{proof}

\begin{theorem} 
If $\cdot$ is associative and commutative, if J is a prime ideal, if $V(J)\neq \varnothing$, and if \{0\} is open in Y then there exist a unique z such that $J=I(z)$.  In this case the I(z) are minimal prime ideals.   Hence if \{z\} is open the ideal is principal (Lemma 63). 
\end{theorem}

\begin{proof}
By Theorem 19 there exist a unique $z_{1}$ such that $J\subseteq I(z_{1})$ and by Theorem 27 there is a $\{z_{2}\}$ such that $I(z_{2})\subseteq J$.  Thus as $I(z_{2})\subseteq J\subseteq I(z_{1})$, $z_{1}=z_{2}$ by Corollary 5 of Definition 10.
\end{proof}

\begin{corollary}
If $\cdot$ is associative and commutative, J is prime, \{0\} is open in Y, and $J\neq I(z)$ for any z then $V(J)=\varnothing$ and $J\nsubseteqq I(z)$.
\end{corollary} 
  
\begin{note}
The functions of $C(Z,\mathbb{Z}_{2})$ are the characteristic functions and the constant functions $\Theta=\chi_{z}$ and $Id=\chi_{\varnothing}$.  The characteristic functions other than $\Theta$ and Id are surjective.  For $\chi\in C(Z,\mathbb{Z}_{2})$ let $U=\chi^{-1}(0)$ so that $\chi=\chi_{u}$.
\end{note}

\noindent  As a quasi-component of $X$ or $Z$ is defined as the intersection of the clopen sets containing the point x or $z=[x]$ and each quasi-component, $[x]$ or $[z]$, corresponds to $I(x)$ or I(z) and hence to the characteristic functions of the clopen sets containing $x$ or z, it is natural to check the relationship between $C(X,Y)$, $C(\Pi,Y)$, C(Z,Y), and $C(Z,\mathbb{Z}_{2})$ where $\mathbb{Z}_{2}$ has the discrete topology.  The classical example is $C(\mathbb{Z},\mathbb{Z}_{2})$ which easily generalizes to $C(Z,\mathbb{Z}_{2})$.  What is important (regardless of the topology of X) because of continuity, the functions of $C(X,\mathbb{Z}_{2})$ are of the form $\chi_u$ where U is a clopen set.  
In Theorems 29 and 30 assume Z has the discrete topology and Y has multiplication and addition with 0 and 1 with 0 open.  Consequently in C(Z,Y), $\chi_{u}+\chi_{u^{c}}=Id$. 

\begin{theorem} 
If I and J are proper prime ideals neither can be a subset of the other.
\end{theorem}

\begin{proof}
For ideals as subsets, if $\chi_{u}$ is in one it must be in the other for $\chi_{u^{c}}$ is excluded from both. 
\end{proof}

\begin{corollary}  
All proper prime ideals are min-max.
\end{corollary}

\begin{note}  
It is known that (for rings)  as $\chi_{u}^n=\chi_{u}$ the prime ideals are max.  
\end{note}

\begin{theorem} 
All proper prime ideals are of the form I(z).
\end{theorem}

\begin{proof}
Let J be a proper prime ideal and $\chi_{u}\in J$ which implies $(\chi_{u})\subseteq J$ and $U\neq \varnothing$. Now choose $z\in U$ so that $(\chi_{u})\subseteq I(z)$.  As $\{z\}$ is clopen $I(z)=(\chi_{z})$.  Note $\chi_{u}\cdot \chi_{z}\in (\chi_{u})\bigcap (\chi_{z})$.  Consequently $\chi_{u}\in (\chi_{z})$ which says $(\chi_{u})\subseteq (\chi_{z})$ and $\chi_{z}\in (\chi_{u})$ which says $(\chi_{z})\subseteq (\chi_{u})$.  Thus $(\chi_{u})=(\chi_{z})=I(z)\subseteq J$.  Therefore, I(z) being a prime ideal, J=I(z).
\end{proof}   

\begin{note}
Theorems 29 and 30 establishes that the proper prime ideals of $C(\mathbb{Z},\mathbb{Z}_{2})$ are all min-max and of the form I(z).  \newline To return to the comparison of the function spaces, assume Z  has a possibly weaker  $\mathcal{T}$ topology and Y has the clopen base topology and $\cdot$ with a 0 and 1.
\end{note}

In consonant with the comment following Theorem 5) on page 8, now set  $[x]_{\chi}=[x]_{2}$ where $[x]_{2}$ is defined:

\begin{defin} 
$[x]_{2}=\{y\mid \forall \chi\in C(Z,\mathbb{Z}_{2}),\chi (y)=\chi (x)\}$ and $\Pi_{2}=\{[x]_{2}\}$.
\end{defin}

Note that $Q_{x}=[x]=[x]_{2}$, (pg 8), so $\Pi =\Pi_{2}= Z$ and consequently $C(X,Y)=C(\Pi,Y)= C(\Pi_{2},Y)=C(Z,Y)$   and $C(X,\mathbb{Z}_{2})=C(\Pi_{2},\mathbb{Z}_{2})=C(Z,\mathbb{Z}_2)$.

First there there is natural injection of $C(Z,\mathbb{Z}_{2})$ into $C(Z,Y)$ if $Y$ has a binary operation with at lease a null element, 0 and unitary element, 1.

\begin{defin} 
Define $j:\mathbb{Z}_{2}\hookrightarrow Y$ by $j(0)=0$ and $j(1)=1$ and define $\mathbb{J}:C(Z,\mathbb{Z}_{2})\rightarrow C(Z,Y)$ by $J(\chi)=j\circ \chi$.
 \end{defin}
  
 \begin{note}
 In the present context $\mathbb{J}$ embeds $C(Z,\mathbb{Z}_{2})$ in C(Z,Y) as $\mathcal{X}$.  If Y is a ring such that 1+1=0  then $\mathbb{J}$ embeds $C(Z,\mathbb{Z}_{2})$ in C(Z,Y) as the subring $\mathcal{X}$, by Theorem 21.
 \end{note}
 
 \begin{theorem} 
 Pulling together results $C(\Pi,\mathbb{Z}_{2})\xrightarrow{G} C(X,\mathbb{Z}_{2})\xrightarrow{\mathbb{J}} C(X,Y)\xrightarrow{H} C(\Pi,Y)$ where the reader can choose the algebraic setting.
 \end{theorem}
 
 Example: $C(\Pi,Z_{2})$ and $C(\Pi,Y)$ need not be of the same cardinality and thus neither algebraically or topologically isomorphic.  Let $X=Y=\mathbb{Z}$ and $\mathbb{Z}_{2}$ all have the discrete topology.  Then $C(\mathbb{Z},\mathbb{Z})=\mathbb{Z}^{\mathbb{Z}}$ and $C(\mathbb{Z},\mathbb{Z}_{2})=\mathbb{Z}^{\mathbb{Z}_{2}}$.  So $\parallel C(\mathbb{Z},\mathbb{Z})\parallel=\aleph_{\circ}^{\aleph_{\circ}}$ and $\parallel C(\mathbb{Z},\mathbb{Z}_{2})\parallel=\aleph_{\circ}^{2}$.  Now observe $\aleph_{\circ}^{2}=\aleph_{\circ}<2^{\aleph_{\circ}}=\aleph_{\circ}^{\aleph_{\circ}}$.  Therefore $C(\mathbb{Z},\mathbb{Z})$ and $C(\mathbb{Z},\mathbb{Z}_{2})$ are not isomorphic in any set based category.
 
 In the other direction there is an interesting algebraic projection  of $C(Z,Y)$ onto $C(Z,\mathbb{Z}_{2})$ when $Y$ has the two binary operation with a unique open null element and unit so that $0\cdot y=0$ and $1\cdot y=y$.
 Recall that the functions of $C(Z,\mathbb{Z}_{2})$ are characteristic functions where $\chi_{Z}=\Theta$ and $\chi_{\varnothing}=Id$. 

 \begin{defin} 
 Define $l:Y\rightarrow\mathbb{Z}_{2}$ by $l(y)=0$ if $y=0$ and $1$ if $y\neq 0$.   
\end{defin}

\begin{lemma}
$l$ is continuous as \{0\} is open in Y.  If $Y$ has no divisors of zero, then $l$ is a surjective multiplicative homomorphism where $l(y_{1}\cdot y_{2})=l(y_{1})\cdot l(y_{2})$.
\end{lemma}

\begin{defin} 
If \{0\} is open in Y define $L:C(Z,Y)\rightarrow C(Z,\mathbb{Z}_{2})$ by $\chi_{f}=L(f)=l\circ f$ so that $\chi_{f}(z)=L(f)(z)=(l\circ f)(z)=l(f(z))$.
\end{defin}

\begin{note} 
First $L(Id)=Id$ and $L(\Theta)=\Theta$.  Second if $f^{-1}(0)=g^{-1}(0)=U$ or $\varnothing$ then $L(f)=L(g)$.  
\end{note}

\begin{theorem} 
If $Y$ has no divisors of zero and \{0\} is open in Y then $L:C(Z,Y)\rightarrow C(Z,\mathbb{Z}_{2})$ is a surjective multiplicative homomorphism.
\end{theorem}

\begin{proof}
First $(L(f_{1}\cdot f_{2}))(z)=(l\circ (f_{1}\cdot f_{2}))(z)=l((f_{1}\cdot f_{2})(z))=l(f_{1}(z)\cdot f_{2}(z))=l(f_{1}))\cdot l(f_{2}(z))$.  To show surjective let $\chi\in C(Z,\mathbb{Z}_{2})$, then $\chi^{-1}(0)$ and $\chi^{-1}(1)$ are clopen sets in $Z$.  So any function from $Z$ to $Y$ that is zero on $\chi^{-1}(0)$ and one on $\chi^{-1}(1)$ is contineous and has as its $L$ image $\chi$. 
\end{proof}

\begin{lemma} 
There is a one to one correspondence between the clopen sets of $Z$ and the functions of $C(Z,\mathbb{Z}_{2})$ where if $U$ is clopen in $Z$ then $\chi_{U}$ is in $C(Z,\mathbb{Z}_{2})$ and if $\chi$ is in $C(Z,\mathbb{Z}_{2})$ then $U=\chi^{-1}(0)$ where $\chi_{U}=\chi$.
\end{lemma}

\begin{note}
$I_{2}(z)$ in $C(Z,\mathbb{Z}_{2})$  is the prime ideal of functions that map $z$ to $0$.  
By Theorems 27) and 28) as all functions of $C(Z,\mathbb{Z})$ are of the form $\chi_{u}$, U clopen, the $I_{2}(z)$ are the min-max prime ideals of $C(Z,\mathbb{Z}_{2})$.  This suggest.
\end{note}

\begin{theorem} 
If $\cdot$ is associative and commutative and $\{0\}$ is open in Y then all proper prime ideals of C(Z,Y) are min-max and of the form I(z), (For CX,Y) of the form $I(Q_{x})$ where $Q_{x}$ is the quasi-component of $x\in X$).
\end{theorem}

\begin{proof}
By Theorem 27) if J is a proper prime ideal then $V(J)\ne \varnothing$.  Thus by Theorem 28) there is a unique z such that J=I(z).  Now if there as a proper prime ideal $K\subsetneq J$ then there is a $z'$ such that $K=I(z')\subset I(z)$ implying K is not prime for its zeros set would have at least two elements.  Analogues if $J\subsetneq K$ then J could not be prime.
\end{proof}

\begin{note}
Again in C(Z,Y) if $\{0\}\subset Y$ is open and $\cdot$ is associative and commutative, then the I(z) are the mini-max prime ideals.      
\end{note} 

\begin{theorem} 
 If $Y$ has no nontrivial left zero-divisors (or $\{0\}$ is open in Y) then the intersection of all prime ideals is zero.  We say the prime radical (or nil radical) is zero although all that is used is that $Y$ has only a single operation with a left zero.
 \end{theorem}
 
 \begin{proof}
 As the intersection of the $I(z)$ over all $z$ in $Z$ is $(\Theta)$, Corollary to  Theorem 14 (or Theorem 33).
 \end{proof}

 \begin{note}
There is a one-to one correspondence between ideals I(z) and the points of $Z$ and hence the quasi-components $X$.  Consequently if $Y$ is a ring and therefore $C(X,Y)$ is a ring, there is such a one-to-one correspondence between quasi-components of $X$ and the prime ring ideals, $I([x])$ of $C(X,Y)$ when $\{0\}$ is open in Y.  
 \end{note}
  
 \begin{lemma} 
 If $Z$ has the discrete topology then $C(Z,Y)$ isomorphic to the categorical product of $Y$ indexed on $Z$ in the category of $Y$.
 \end{lemma}
  
 EXAMPLES: Let $Y$ be a ring with unit,  $Z$ a set with the discrete topology, and 
 \newline  $n=\parallel Z \parallel<\aleph_0$.  Let $Z=\{z_i: i=0, 1, ...., n\}$.  Define $\chi_i(z)=\left\{\begin{array}{ll}1&\text{if } z=z_i \\0&\text{if } z\ne z_i\end{array}\right.$.  So $\chi_i$ is in $C(Z,Y)$.  For each i set $f_i=\chi_{i}\cdot f$.    Then $f(z)=\sum\limits_{i=0}^n(\chi_{i}\cdot f_{i})(z)$ for any $z$ and $f\in C(Z,Y)$.  If I is an ideal and $f\in I$ then $\chi_{i}\cdot f\in I$.  Thus the elements of I have a representation of this form.  All such rings should be well known.  If $\parallel Z\parallel \geqq \aleph_0$, we still have categorical products for which the same results hold with the Axion of Choice as $C(Z,Y)$ is a function space.  All of these rings have the prime radical zero and the I(z) are minimal prime ideals.
 
 If however not all of the points of Z are clopen then the ring C(Z,Y) is no longer a product.  Consider for an example all points clopen except one.
 
Aside:  For a ring of functions the set of ring ideals of the form $I(U)$ where $U$ is a clopen set form a complete modular lattice using intersection and addition.
  
Now assume that $Y$ will be a ring and that the objects being discussed are ring objects.
															
\begin{theorem} 
 If I is a proper ideal, $f\in I$ such that $f\neq \Theta$, and U is a clopen set then the the principal ideals generated by $f\cdot \chi_u$ and $f\cdot (1-\chi_u)$ are sub ideals of I.
 \end{theorem}
 
\begin{theorem}
If z is not open then for every clopen set U containing z there is a $z_u\in U$ such that $z_u\neq z$.   z is the unique cluster point of $\{z_u\}$ where $\{z\}= \cap\{U: U\mbox{is clopen and}\  z\in U\}$.
\end {theorem}

\begin{proof}
As z is not open, $\{z\}\neq U$.  Suppose $z_1\neq z$ is another cluster point of $\{z_u\}$.  But z and $z_1$ can be separated by clopen sets.  
\end{proof}

\begin{lemma} 
If U is clopen then $I(U)=(\chi _u)$ and $I(U^c)=(\chi_{u^c})= (1-\chi_u)$.  Therefore $I(U)$ and $I(U^c)$ are coprime (comaximal).  Here $U^c=Z\setminus U$ is the Z complement of U.
\end{lemma}

\begin{corollary}
If z is open then $I(z)=(\chi_z)$ and $I(z^c)=(\chi_{z^c})$ are co-prime.
\end{corollary}

\begin{lemma} 
If Y has no divisors of zero then V((f))=V(f), cf. Lemma 40.
\end{lemma}

\begin{proof}
If $g\in (f)$ then $g=h\cdot f$ for some $h\in C(Z,Y)$.  So $V(g)=V(h\cdot f)=V(h)\bigcup V(f)$ which implies V((f))=V(f).          	
\end{proof}
 
\begin{lemma}
If $f\in I(z)$ and $f^{-1}(0)\neq \{z\}$ then $I(z)\neq (f)$ although $(f)\subset I(z)$.
\end{lemma}

\begin{proof}
By Lemma 40 or 70, $V((f))=V(f)=f^{-1}(0)\neq \{z\}=V(I(z))$ so $(f)\neq I(z)$.
\end{proof}

\begin{lemma}
$I(\cap \{U: z\in U\})=I(z)$.
\end{lemma}

\begin{lemma} 
$\cap \{I(A_\alpha): \alpha\in \Lambda\}=I(\cup \{A_\alpha:  \alpha \in \Lambda\})$.
\end{lemma}

\begin{corollary}
For any set U, $\cap\{I(z): z\in U\}=I(U)$.  Hence $\cap \{I(z):z\in Z\}=I(Z)=(\Theta)$.
\end{corollary}

\noindent  A Summary:  

C(Z,Y) is not a local ring Lemma 59, no. 1 and [1], pg. 11 exercise 12.  Spect(C(Z,Y)) is disconnected topological space, [1], pg. 14, exercise 22.  If Y has no one or two sided divisors of zero (eg. an integral domain), then the prime radical (nillradical) of C(Z,Y) is trivial, Corollary to Theorem 14.

Now to continue, from proposition 1.11 [ 1 ] page 8, we obtain.

\begin{lemma} 
Let Y be an integral domain.  \newline
a).   If an ideal $I\subset \bigcup_{1\leq i\leq n}I(z_i)$ then $I\subseteq I(z_i)$ some i.  \newline
b).   Given $\{z_i\}^{n}_1$ and $I(\cup_i\{z_i\})=\bigcap_iI(z_i)\subset p$.  If p is prime then $I(z_i)\subseteq p$ for some i.  If $p=\bigcap_iI(z_i)$, then $p=I(z_i)$ for some i, $1\leq i\leq n$.
\end {lemma}

\begin{theorem} 
Let I be a proper ideal, $f\in I$ such that $f\neq 0$, and U a clopen set.  Then the principal ideals $(f\cdot \chi_u)$ and $(f\cdot \chi_{u^{c}})$ are sub ideals of I.
\end{theorem}

\begin{lemma}
If f is a function, U a clopen set, I an ideal, and if $f\cdot \chi_{u}$ and $f\cdot \chi_{u^{c}}$ are both elements of I, then $f=f\cdot \chi_{u} +f\cdot \chi_{u^{c}}$ is an element of I.
\end{lemma} 

\begin{theorem} 
If I is a prime ideal, $f\notin I$, and U a clopen set, then one and at most one of $(f\cdot \chi_{u})$ and $(f\cdot \chi_{u^{c}})$ is a principal sub ideal of I.
\end{theorem}

\begin{proof}
At least one of $\chi_u$ and $\chi_{u^{c}}$ is in I.
\end{proof}

\begin{lemma}
  Let $I_1, I_2, \cdots I_n$ be prime ideals and $I$ an ideal such that $I\not\subset I_i$ and $I\neq I_i$ for 
  $i=1,2,\cdots,n$, then there exists $a\in I$ such that $a\notin I_i$ for $i=1,2,\cdots,n$. (This is well known to ring theorists)
\end{lemma}

\begin{theorem} 
  If $Y$ is a division ring and $Z=\bigcup\limits_{i=1}^n\{z_i\}$ then the $I(z_i)$ are the maximal ideals.  Hence C(Z,Y) is semi-local, definition on pg. 4 of [1].
\end{theorem}

\begin{proof}
Note every {z} is clopen.  Let I be a maximal ideal such that $I\not= I(z_i)$ for any $i=1, 2,...,n$.  Therefore for each i there exist $f_i$ in I such that $f_i(z_i)\not= 0$.  Thus $\hat{f_i}(\xi)=\left\{\begin{array}{ll}f_i(\xi)&\text{if }\xi=z_i \\0&\text{if } \xi\not= z_i\end{array}\right.$ is continuous and in I as it is $f_i$ times a characteristic function one on $z_i$.  Therefore $f(z)=\sum\limits_{i=1}^n \hat{f_i}(z)\not= 0$ is continuous in I and  not zero for all $z\in Z$.  Hence I is the ring C(Z,Y).
 \end{proof}
 
\begin{note}
By Theorems 24 and 25, if I(z) is not an minimal prime ideal then \{z\} is not open in Z and \{0\} is not open in Y.  Significantly if \{0\} is open in Y then all of the I(z) are minimal.  This suggest interesting examples.  Let Z be any set supporting a topology where all 
points are clopen except a distinguished point denoted d which is closed but not open.  Denote this topological space by $Z_c$.  Let $Z_d$ denote the topological space consisting of the same set with the discrete topological.  Let $\iota:Z_d\rightarrow Z_c$ be the set identity function, $\iota (z)=z$, which is continuous.  Define $\mathfrak{I}:C(Z_c,Y)\rightarrow C(Z_d,Y)$ by $\mathfrak{I}(f)=f\circ \iota$.  $\mathfrak{I}$ embeds $C(Z_c,Y)$ algebraic isomorphically into $C(Z_d,Y)$ but not onto as $\chi_{d}\in C(Z_d,Y)$ and $\chi_{d}\notin C(Z_c,Y)$.  The ideals 
of $C(Z_{c},Y)$ do not map to ideals in $C(Z_{d},Y)$.  Let $I_c$ denote an ideal in $C(Z_{c},Y)$ and $I_d$ an ideal in $C(Z_{d},Y)$.  Note $\mathfrak{I}(I_{c}(z))\subset I_{d}(z)$ for some $I_{d}(z)$ and $\mathfrak{I}(I_{c}(d))\neq I_{d}(d)=(\chi_{d})$ as $\chi_{d}\notin I_{c}(d)$.  $\mathfrak{I}(I_{d}(d)$ is not prime although $\{d\}$ is open in $Z_{d}$ and $I_{d}(d)$ is a minimal prime ideal.  All ideals of both rings of functions that vanish at a point are minimal prime except $I_{c}(d)$ which is prime in $C(Z_c,Y)$ but may not be minimal unless \{0\} is open in Y.  Note $\mathfrak{I}(I_{c}(d)$ is not prime which 
is  a sub ideal of the minimal prime ideal $I_{d}(d)$ in $C(Z_d,Y)$.  \newline  For example if $Y=\mathbb{Z}_2$, then $I_{c}(d)$ is a minimal prime ideal in $C(Z_{c},\mathbb{Z}_{2})$.  But in this example there is no function f such that $V(f)=f^{-1}(0)=d$.  See 6) below.
\end{note}

\begin{note} 
Considerations to use for examples:  \newline  Let Z be a topological space whose points are quasi-components.  See Theorem 5 and the second paragraph of the Note following the Comment after Theorem 6. \newline  Let Y be a $T_0$ topological space with a basis of clopen sets. . See paragraph before the second Note following Definition 5.  \newline  Consider C(Z,Y).  In the text the algebraic conditions on Y are successively stronger.  For the examples assume that Y is a ring although this can be weakened. \newline 1)  A point $z\in Z$ is isolated iff there exists an open point $\{y\}\subset Y$ and a function $f\in C(Z,Y)$ such that $f^{-1}(y)=x$.  \newline 2)  By Theorem 14 if Y has no divisors of zero the I(z) are prime ideals.  \newline  3)  By Theorem 16 or 24, if  \{z\} is open then I(z) is a minimal prime ideal.  \newline  4)  By Theorems 14, 27, and 28 if \{0\} is open in Y then for every z, I(z) is a minimal prime ideal and there are no other minimal prime ideals.  \newline  5)  By Lemma 68, if Z has the discrete topology then C(Z,Y) is algebraically isomorphic to the categorical product of Y indexed on Z in the category of Y.  \newline 6)  To complement 4) above consider C(Z,Y) where $Z=\{0\}\bigcup \{1/n\vert n=1,2,...\}$ has the topology such that all points of Z are open except $\{0\}$ which is closed but not open.  Choose the rationales for Y with the topology such that $\{0\}$ is open.  Then the function$f=\{(0,0)\}\bigcup \{(y,y) \vert y=1/n, n=1,2,...\}$ is not continuous.  The functions of I(0) must be eventually zero.
\end{note}

The following observation and question indicates that our results are not the strongest possible, perhaps foolish, or that a critical example supporting this development is difficult to construct.  

Let $\{z_0\}$ be a quasi-component of Z possibly not open.  Assume \{0\} in Y is not open and let $\mathfrak{U}=\{U:z_0\in U \and\ U open \}$, so $\bigcap \mathfrak{U}=\{z_0\}$.  If $f\in I(z_0)$ then $f=\bigcup \{f\cdot \chi_U:U\in\mathfrak{U}\}$.  Suppose $J\subseteq I(z_0)$ is a prime ideal.  Then $f\cdot \chi_U\in J$ but V(f) need not be copen.  Can this imply that $f\in J$, that is $J=I(z_0)$.  If $\{z_0\}$ is clopen, $J=I(z_0)$ but if $\{z_0\}$ is not open $I(z_0)$ may not be a minimal prime ideal.

\noindent {\bf Examples:}\\
As a first example consider $X=\{(0,0)\}\cup\{(0,1)\}\cup\{(1/n,y):0\leq y\leq 1$ and $n\in \mathbb{N}\}$, $\mathbb{N}$ 
the positive integers, as a subset of the plane.  Then the union of the components $\{(0,0)\}$ and $\{(0,1)\}$ is the quasi-component $\{(0,0),(0,1)\}$.

Now in general each point of a space lies in a component, the maximal connected set containing the point, and lies in its quasi-component, an intersection of clopen sets, which need not be connected.   Thus quasi-components are a union of components.

Open components are quasi-components which are thus clopen.  But quasi-components need not be open.  If components are open then they are quasi-components of $X$.

Let the components of $X$ be open.  Thus $X/\Pi$ has the discrete topology.  If the set $\{[x]\}$ has the discrete topology then $C(X,Y)$, $C(Z,Y)$, $C(X/\Pi,Y)$, and $C(\{[x]\},Y)$ are isomorphic.  Assuming Well Ordering, if $\kappa$ is the cardinality of $\{[x]\}$ and $\kappa$ has the discrete topology then these rings are isomorphic with $C(\kappa, Y)$ which is a product.

If $X$ has only one quasi-component, $[x]$, that is not a component then the components in the complement of $[x]$ are clopen quasi-components.  If $C(\kappa,Y)$ is as above then there is an injective homorphism of $C(X,Y)$ strictly embedding it into $C(\kappa, Y)$.  As $[x]$ is the intersection of clopen sets and is not itself open, each of these clopen sets defining [x] must intersect a quasi-component other than $[x]$.  Using these observations identify $C(\kappa,Y)$ with all sequences from $\kappa$ to $Y$ and $C(X,Y)$ with those that are continuous at $[x]$.  As above consider again $X=\{(0,0)\}\cup\{(0,1)\}\cup\{(1/n,y): 0\leq y\leq 1$ and $n\in \mathbb{N}\}$, $\mathbb{N}$ the positive integers, as a subset of the plane and $Y=\mathbb{Z}_2$, the ring of 0 and 1.  Here $C(\kappa,\mathbb{Z}_2)$ is equinumerous with the set of functions from $\mathbb{N}$ to $\mathbb{Z}_2$ of cardinality $2^{\aleph_0}$ which is larger than $\aleph_0$.  As each functions in $C(X,\mathbb{Z}_2)$ is eventually constant $C(X,\mathbb{Z}_2)$ injects into the set of all finite sequences of $\mathbb{N}$ into $\mathbb{Z}_2$ which is of cardinality less than or equal to $\aleph_0$.  Thus the cardinality of $C(\kappa, \mathbb{Z}_2)$ is strictly larger than that of $C(X,\mathbb{Z}_2)$.  Consequently they can not be algebraically isomorphic. 

REFERENCES

[1]  M. F. Atiyah and I. G. McDonald, Introduction to Commutative Algebra,  Addison-Wesley, 1969.

[2]  N. Bourbaki, Algebra, Addison-Wesley, 1974.

[3]  N. Bourbaki, Commutative Algebra, Addison-Wesley, 1972.

[4]  Eduard $\hat{C}$ech, Topological Spaces, Interscience Pub., 1966.

[5]  R.  Engelking, Outline of General Topology, North Holland Pub. Co., 1968.

[6]  W.  Fulton, Algebraic Curves, W. A. Benjamin Inc., 1969.

[7]  John L. Kelley, General Topology, D. Van Nostrand Co., 1955.

[8]  Joachim Lambek, Lectures on Rings and Modules, Blaisdell Pub. Co.,1966.

[9]  Karen Smith, et all, An Invitation to Algebraic Topology, Springer-Verlog, 2000

[10]  Lynn A. Steen and J. Arthur Seebach, Jr.,Counterexamples in Topology, Holt, 
      Rinehart \indent \quad and Winston, Inc., 1970.

[11]  Stephen Willard, General Topology, Addison-Wesley Pub. Co., 1968.

\end{document}